\documentclass[11pt,dvipdfmx]{article}
\usepackage{amsmath,amssymb,amsthm}
\usepackage{ascmac}
\usepackage{color}
\usepackage{hyperref}
\usepackage[english]{babel} 
\usepackage{authblk} 
\usepackage{tikz}
\usepackage{enumitem}

\theoremstyle{definition}
\newtheorem{theorem}{Theorem}[section]
\newtheorem{remark}[theorem]{Remark}
\newtheorem{lemma}[theorem]{Lemma}

\newtheorem{corollary}[theorem]{Corollary}
\newtheorem{definition}[theorem]{Definition}
\newtheorem{question}[theorem]{Question}

\newtheorem*{proof of claim}{Proof of Claim}
\newtheorem{observation}[theorem]{Observation}

\newcounter{claimcounter}
\newtheorem*{claim}{Claim \thetheorem.\theclaimcounter}

\newcommand{\Sat}{\mathsf{Sat}}

\DeclareMathOperator{\RFN}{\mathsf{RFN}}
\newcommand{\ATR}{\mathsf{ATR}}
\newcommand{\ATRo}{\mathsf{ATR}_0}
\newcommand{\ACAo}{\mathsf{ACA}_0}
\newcommand{\RCAo}{\mathsf{RCA}_0}
\newcommand{\CAo}{\mathsf{CA}_0}
\newcommand{\WKLo}{\mathsf{WKL}_0}

\DeclareMathOperator{\TI}{\mathsf{TI}}

\DeclareMathOperator{\ACA}{\mathsf{ACA}}

\DeclareMathOperator{\LPP}{\mathsf{LPP}}
\newcommand{\TLPP}{\mathsf{TLPP}}
\newcommand{\ALPP}{\mathsf{ALPP}}

\DeclareMathOperator{\dom}{dom}

\DeclareMathOperator{\HJ}{HJ}

\DeclareMathOperator{\WO}{WO}

\newcommand{\T}{\mathrm{T}}
\renewcommand{\a}{\mathrm{a}}
\newcommand{\It}{\mathsf{It}}
\newcommand{\Ram}{\mathsf{Ram}}
\newcommand{\Det}{\mathsf{Det}}
\DeclareMathOperator{\rel}{rel}
\newcommand{\N}{\mathbb{N}}

\renewcommand{\L}{\mathcal{L}}

\newcommand{\M}{\mathcal{M}}
\renewcommand{\phi}{\varphi}
\newcommand{\myhyphen}{\mathchar`-}

\newcommand{\even}{\mathrm{even}}
\newcommand{\odd}{\mathrm{odd}}
\DeclareMathOperator{\Con}{\mathsf{Con}}

\newcommand{\A}{\mathsf{A}}
\newcommand{\NN}{\mathcal{N}}
\newcommand{\HH}{\mathcal{H}}
\newcommand{\RP}{\mathrm{RP}}
\newcommand{\pRP}{\mathrm{pRP}}
\newcommand{\hyp}{\mathchar`-}

\newcommand{\qedclaim}{
  \renewcommand{\qedsymbol}{$\blacksquare$}
  \qed
  \renewcommand{\qedsymbol}{$\box$}
}

\author[1]{YUDAI SUZUKI}
\author[2]{KEITA YOKOYAMA}
\affil[1]{National Institute of Technology, Oyama College, Tochigi, Japan}
\affil[2]{Mathematical Institute, Tohoku University, Sendai, Japan}
\affil[1]{yudai.suzuki.q1@dc.tohoku.ac.jp}
\affil[2]{keita.yokoyama.c2@tohoku.ac.jp}

\title{On the $\Pi^1_2$ consequences of $\Pi^1_1\mathchar`-\mathsf{CA}_0$}
\date{\today}

\begin{document}

\maketitle

\begin{abstract}
  In this paper, we introduce a hierarchy dividing the set $\{\sigma \in \Pi^1_2 : \Pi^1_1\myhyphen\CAo \vdash \sigma\}$.
  Then, we give some characterizations of this set using weaker variants of some principles equivalent to $\Pi^1_1\myhyphen\CAo$: leftmost path principle, Ramsey's theorem for $\Sigma^0_n$ classes of $[\N]^{\N}$ and determinacy for $(\Sigma^0_1)_n$ classes of $\N^{\N}$.
\end{abstract}

\section{Introduction}
  Reverse mathematics is a program to classify theorems in various fields of mathematics according to their logical strength.
  In the most typical study of this area, one uses the $\textit{big five}$ of axiomatic systems of second-order arithmetic, $\RCAo$, $\WKLo$, $\ACAo$, $\ATRo$ and $\Pi^1_1\myhyphen\CAo$.
Undeniably, many theorems in the ``core-of-mathematics'' are shown to be equivalent to one of the big five (see, e.g., for \cite{Simpson, DzMu}).

Among the big five systems, the strongest system, $\Pi^1_1\myhyphen\CAo$, is considerably different from others; all other systems are axiomatized by $\Pi^{1}_{2}$ sentences while $\Pi^1_1\myhyphen\CAo$ is never implied from a true $\Pi^{1}_{2}$ sentence.
In the study of reverse mathematics, there are several mathematical theorems which are formalized by $\Pi^{1}_{2}$ sentences and known to be provable from $\Pi^1_1\myhyphen\CAo$ such as Kruskal's tree theorem, Nash-Williams' theorem, Menger's theorem, Fra\"iss\'e's conjecture and Caristi-Kirk fixed point theorem (see \cite{marcone_leftmost, pakhomov_solda_nash, Shafer_Menger, Montalban-Fraisse, Freund-Fraisse, F-DSTY} for the recent progress on these studies\footnote{It is shown that Nash-Williams' theorem is provable from $\ATRo$ in a recent paper by
Pakhomov and Sold{\`a} \cite{pakhomov_solda_nash}.}).
Then a natural question arises: \textit{how can we calibrate the strength of consequences of $\Pi^1_1\myhyphen\CAo$ which are of complexity $\Pi^{1}_{2}$?}

  In \cite{Townser_TLPP}, Towsner introduced the \textit{relative leftmost path principles} to give a new upper bound
  for theorems located between $\ATRo$ and $\Pi^1_1\myhyphen\CAo$.
  He focused on an equivalent of $\Pi^1_1\myhyphen\CAo$ called the leftmost path principle which states that
  any ill-founded tree has a leftmost path.
  His argument is based on the following idea. Since the leftmost path principle is a $\Pi^1_3$ statement,
  the actual leftmost path is not needed to prove a $\Pi^1_2$ sentence. Instead of the leftmost path, it is enough to use a path which behaves as a leftmost path in a certain range.
  Such a path is called a relative leftmost path.
  Indeed, the hierarchy of the relative leftmost path principles already capture many of the $\Pi^{1}_{2}$ consequences of $\Pi^1_1\myhyphen\CAo$ which are mentioned above. On the other hand, this hierarchy does not cover the whole $\Pi^{1}_{2}$ consequences of $\Pi^1_1\myhyphen\CAo$.

  In this paper, we extend the idea of \textit{relativization of $\Pi^1_3$ statements} to the level of $n$-th hyperjump.
  We define $\beta^1_0\RFN(n)$ as the assertion that
  for each set $X$, there is a coded $\omega$-model $\M$ of $\ACAo$ such that
  $X \in \M$ and $\M \models \exists Y(Y = \HJ^n(X))$. Then a witness $Y$
  works as a good approximation of the actual $n$-th hyperjump.
  We show that a $\Pi^1_2$ sentence provable from $\Pi^1_1\myhyphen\CAo$ is already provable from $\beta^1_0\RFN(n)$ for some $n$.

  Intuitively, a sentence $\sigma$ is provable from $\beta^1_0\RFN(n)$ means that
  there is a proof of $\sigma$ from $\Pi^1_1\myhyphen\CAo$ such that
  the use of the hyperjump operator in it is up to $n$-times.
  In this sense, Towsner's transfinite leftmost path principle $\TLPP$ is in the level of
  single use of the hyperjump operator.
  Specifically, we show that there is a variant of the relative leftmost path principle which is equivalent to $\beta^1_0\RFN(1)$, and
  a variant of $\beta^1_0\RFN(1)$ is enough to prove $\TLPP$.
  We also introduce an iterated form of the relative leftmost path principles which is in the level of $\beta^1_0\RFN(n)$.
  We note that the first author shows that both of Towsner's relative leftmost path principles and $\beta^1_0\RFN(1)$ (including some variants of it) 
  can be characterized as a variant of the $\omega$-model reflection of transfintie induction in a forthcoming paper \cite{suzuki2024relative}\footnote{A characterization of $\beta^1_0\RFN(1)$ via the $\omega$-model reflection of transfinite induction is independently given by Freund \cite{Freund-Fraisse}.}.
  In this sense,  our work can be seen as a natural extension of Towsner's work reaching the whole set of the $\Pi^1_2\hyp$consequences of $\Pi^1_1\hyp\CAo$.

  We apply the idea of relativization to certain descriptive set theoretic principles: Ramsey's theorem for $[\N]^{\N}$ (known as Galvin-Prikry's theorem) and the determinacy for $\N^{\N}$.
  It is known that both of Ramsey's theorem for $\Sigma^0_n$ classes and the determinacy for $(\Sigma^0_1)_n$ classes are equivalent to
  $\Pi^1_1\myhyphen\CAo$ for $n >1$.
  In this paper, we show that the $\Pi^1_2$ consequences of $\Pi^1_1\myhyphen\CAo$ coincides with the theories
  $\ACAo + \{\rel(\Sigma^0_n\myhyphen\Ram) : n \in \omega\}$ or $\ACAo + \{\rel((\Sigma^0_1)_n\myhyphen\Det) : n \in \omega\}$,
  where $\rel(\sigma)$ denotes the relativization of $\sigma$.

  \subsection*{Acknowledgements}
  The authors thank Anton Freund for helpful discussions and comments.
  The first author's work is supported by JST, the establishment of university
  fellowships towards the creation of science technology innovation,
  Grant Number JPMJFS2102.
  The second author's work is partially supported by JSPS KAKENHI grant numbers JP19K03601, JP21KK0045 and JP23K03193.

\section{Preliminaries}
This section introduces the basic notions of reverse mathematics. 
For the details, see also Simpson's textbook \cite{Simpson}.
As usual, we denote the set of natural numbers by $\omega$, and use $\N$ to represent the range of number variables in the language $\L_2$ of second-order arithmetic.


Reverse mathematics is a research program that classifies mathematical theorems according to their logical strength. To achieve this, we identify the logical strength of a theorem with the weakest axiom needed to prove it.  Most research in reverse mathematics is based on second-order arithmetic\footnote{Sometimes the phrase `second-order arithmetic' means a specific axiomatic system $\mathsf{Z}_2$. However, in this paper, we use this phrase as a general term for axiomatic systems of the language of second-order arithmetic.}.  Specifically, for a given formula $T$ in second-order arithmetic representing a mathematical theorem, we seek an axiomatic system $\Gamma$ and an axiom $A$ such that:
\begin{itemize}
  \item $\Gamma$ is strong enough to interpret $T$, but may not be strong enough to prove it,
  \item $A$ is equivalent to $T$ over $\Gamma$.
\end{itemize}
We begin with introducing some axiomatic systems of second-order arithmetic.
\begin{definition}
  Let $\theta(y,\vec{x},\vec{X})$%
  \footnote{Here, $\vec{x}$ and $\vec{X}$ are abbreviations of $x_1,\ldots,x_n$ and $X_1,\ldots,X_m$. When the numbers $n$ and $m$ of variables are not important, we use these notations.}
  be a formula with exactly displayed free variables.
  Then, $\theta$-comprehension axiom is the formula
  \begin{align*}
    \forall \vec{x},\vec{X} \exists Y \forall y( y \in Y \leftrightarrow \theta(y,\vec{x},\vec{X})).
  \end{align*}
  Intuitively, comprehension for $\theta$ states that the set $\{y : \theta(y)\}$ exists.

  Some $\L_2$-theories are characterized by comprehension schemas.
  \begin{itemize}
    \item $\RCAo$ consists of `$(0,1,+,\cdot,<)$ forms a discrete ordered semi-ring', induction for $\Sigma^0_1$ formulas and comprehension for $\Delta^0_1$ formulas, \textit{i.e.}, $\theta$-comprehension is applicable when $\theta$ is a $\Sigma^0_1$ formula which is equivalent to a $\Pi^{0}_{1}$ formula.
    \item $\ACAo$ consists of $\RCAo$ and comprehension for $\Sigma^1_0$ formulas.
    \item $\Pi^1_1\myhyphen\CAo$ consists of $\RCAo$ and comprehension for $\Pi^1_1$ formulas.
  \end{itemize}
\end{definition}

In $\RCAo$, the Turing jump operator and its iteration can be defined.
For a given (countable) ordinal $\alpha$, we write $X^{(\alpha)}$ to mean
the $\alpha$-times iteration of Turing jump at $X$.
For the details, see section VIII.1 in \cite{Simpson}.

\begin{definition}
  We define $\ACA_0'$ and $\ACA_0^+$ as follows.
  \begin{itemize}
    \item $\ACA_0'$ consists of $\RCAo$ and $\forall n \forall X \exists Y (Y = X^{(n)})$.
    \item $\ACA_0^+$ consists of $\RCAo$ and $\forall X \exists Y (Y = X^{(\omega)})$.
    \item $\ATRo$ consists of $\RCAo$ and the statement `for any ordinal $\alpha$, $Y = X^{(\alpha)}$ exists'.
  \end{itemize}
\end{definition}

It is known that $\ACAo < \ACA_0' < \ACA_0^+ <\ATRo$.

\subsection{Trees}
We introduce some notions on trees.
Then we give a characterization of $\Pi^1_1\myhyphen\CAo$ by trees.

\begin{definition}
  Let $X$ be a set.
  We write $X^{<\N}$ for the set of finite sequences of elements of $X$ and $X^{\N}$ for the set of infinite sequences of elements of $X$.
  More precisely, each $\sigma \in X^{<\N}$ is a function whose domain is $\{0,1,\ldots,n-1\}$ for some $n$ and $\sigma(i)$ is an element of $X$ for any $i <n$ . Similarly, each $f \in X^{\N}$ is a function whose domain is $\N$ and whose codomain is $X$.

  For each  $\sigma \in X^{<\N}$, define the length $|\sigma|$ by the number $n$ such that $\dom(\sigma) = \{0,1,\ldots,n-1\}$.
\end{definition}
Intuitively, we identify $\sigma \in X^{<\N}$ and the sequence $\langle \sigma(0),\ldots,\sigma(|\sigma|-1) \rangle$,
$f \in X^{\N}$ and the sequence $\langle f(n) \rangle_{n \in \N}$.
\begin{definition}
  Let $\sigma,\tau$ be finite sequences.
  \begin{itemize}
    \item Define the concatenation $\sigma \ast \tau$ by
    $\langle \sigma(0),\ldots,\sigma(|\sigma|-1),\tau(0),\ldots,\tau(|\tau|-1) \rangle$.
    \item We say $\sigma$ is an initial segment of $\tau$ (write $\sigma \preceq \tau$) if
    $|\sigma| \leq |\tau|$ and $\forall i < |\sigma| (\sigma(i) = \tau(i))$. We say $\sigma$ is a proper initial segment of $\tau$
    (write $\sigma \prec \tau$) if $\sigma \preceq \tau$ and $\sigma \neq \tau$.
    \item Let $f$ be an infinite sequence. We say $\sigma$ is an initial segment of $f$ (write $\sigma \prec f$) if
    $\forall i < |\sigma| (\sigma(i) = f(i))$.
  \end{itemize}
\end{definition}

\begin{definition}
  Let $T$ be a subset of $X^{<\N}$. We say $T$ is a tree on $X$ if
  $T$ is closed under taking initial segments.

  Let $T \subseteq X^{<\N}$ be a tree.
  A function $f \in X^{\N}$ is called a path of $T$ if $(\forall n)(f[n] \in T)$. Here, $f[n]$ denotes the initial segment of $f$ with length $n$.
  The set of all paths of $T$ is denoted by $[T]$.
  We say $T$ is ill-founded if $[T] \neq \varnothing$.
\end{definition}

\begin{lemma}
  (\cite{Simpson}, III.7.2)
  $\ACAo$ proves the following K\H{o}nig lemma.
  Let $T$ be an infinite tree. If $\forall \sigma \in T \exists m \forall i (\sigma \ast \langle i \rangle \in T \to i < m)$, then $T$ is ill-founded.
\end{lemma}

\begin{lemma}
  \cite[VI.1.1]{Simpson} The following are equivalent over $\RCAo$.
  \begin{enumerate}
    \item $\Pi^1_1\myhyphen\CAo$,
    \item for any sequence $\langle T_k \rangle_k$ of trees, there is a set $X = \{k : [T_k] \neq \varnothing\}$.
  \end{enumerate}
\end{lemma}

Following the above lemma, we give a new characterization of $\Pi^1_1\myhyphen\CAo$ by using the notion of leftmost path.
\begin{definition}
  ($\RCAo$)
  Let $f,g \in \N^{\N}$. We write $f <_l g$ if $\exists n (f[n] = g[n] \land f(n) < g(n))$ and $f \leq_l g$ if $f <_l g \lor f = g$.
  This $<_l$ forms a total order on $\N^{\N}$ and is called the lexicographical order.

  Let $T$ be a tree and $f \in [T]$. We say $f$ is the leftmost path of $T$ if $\forall g \in [T](f \leq_l g)$.
\end{definition}

In the remaining of this section, we show that the equivalence of the existence of a leftmost path and $\Pi^1_1\myhyphen\CAo$.
The key is that a sequence of trees can be coded by a tree.

\begin{definition}\label{oplus of trees}
  ($\RCAo$)
  Let $\sigma \in \N^{<\N}$. For each $l < |\sigma|$, define $n_l$ as the maximum $n$ such that $(n,l) < |\sigma|$. Then we define  $\sigma_l = \langle \sigma((0,l)), \ldots, \sigma((n_l,l)) \rangle$.
  Thus, each $\sigma \in \N^{<\N}$ is regarded as a sequence $\langle \sigma_l \rangle_{l < |\sigma|}$ of sequences.

  Conversely, for each sequence $\langle \sigma_l \rangle_{l < L}$ of sequences, define $\bigoplus_{l < L} \sigma_l$ by
  \begin{align*}
    (\bigoplus_{l < L} \sigma_l)(n,l) = \sigma_l(n).
  \end{align*}
  For each sequences $\langle f_l \rangle_{l \in \N}$ of functions, define $\bigoplus_{l} f_l$ by
  \begin{align*}
    (\bigoplus_{l < L} f_l)(n,l) = f_l(n).
  \end{align*}

  Finally, for each $\langle T_l \rangle_{l \in \N}$ of trees, define $\bigoplus_{l} T_l$ by
  \begin{align*}
    \{\bigoplus_{l < L} \sigma_l : \langle \sigma_l \rangle_{l < L} \in T_0 \times \cdots \times T_{L-1}\}.
  \end{align*}
\end{definition}
\begin{lemma}
  $(\RCAo)$ Let $\langle T_l \rangle_l$ be a sequence of trees.
  Then, the operator $\bigoplus_l$ is a bijection between $\prod_l [T_l]$ and $[\bigoplus_l T_l]$.
  Moreover, this operator preserves the lexicographical order in the sense that for any $\langle f_i \rangle_i, \langle g_i \rangle_i \in \prod_i [T_i]$,
  \begin{enumerate}
    \item $\forall i (f_i \leq_l g_i)  \to \bigoplus_i f_i \leq_l \bigoplus_i g_i $,
    \item $\bigoplus_i f_i$ is the leftmost path of $\bigoplus_i T_i$ if and only if each $f_i$ is the leftmost path of $T_i$.
  \end{enumerate}
\end{lemma}

\begin{definition}
  Let $T$ be a tree.
  We say $T$ is pruned if $\forall \sigma \in T \exists n (\sigma\ast\langle n \rangle \in T)$.
\end{definition}
\begin{lemma}
  ($\RCAo$)
  Let $T$ be a pruned tree and $\sigma \in T$.
  Then there is a $T$-computable path $f$ such that $\sigma \prec f$.
  Moreover, we can take this $f$ to be the leftmost of all paths extending $\sigma$.
\end{lemma}
\begin{proof}
  It is immediate from the definition.
\end{proof}

\begin{theorem}\label{LPP and Pi11CA}
  \cite[Theorem 6.5.]{marcone_leftmost}
  The following are equivalent over $\RCAo$.
  \begin{enumerate}
    \item $\Pi^1_1\myhyphen\CAo$,
    \item Every ill-founded tree has the leftmost path.
  \end{enumerate}
\end{theorem}
We review the proof of this theorem with the notions defined above since we will refine the argument here in the later discussions.
\begin{proof}
  $(1 \to 2)$. Let $T$ be an ill-founded tree. Define $S = \{\sigma \in T: \exists f \in [T] (\sigma \prec f)\}$.
  Then $S$ is a pruned tree such that $[S] = [T]$. Since $S$ is pruned, $S$ has the leftmost path.
  Clearly this is also the leftmost path of $[T]$.

  $(2 \to 1)$ It is enough to show that the second clause implies for any sequence $\langle T_k \rangle_k$ of trees, there is a set $X = \{k : [T_k] \neq \varnothing\}$.

  Assume that each ill-founded tree has the leftmost path. Let $\langle T_k \rangle_k$ be a sequence of trees.
  For each $k$, define $S_k$ by $\{\langle 0 \rangle \ast \sigma : \sigma \in T_k \} \cup \{1^l : l \in \N\}$ where $1^l$ is the sequence with length $=l$ such that each element of it is $1$.
  Then, each $S_k$ has a path $1^{\infty} = \langle 1,1,\ldots,\rangle$.
  Thus, $\bigoplus_k S_k$ is also ill-founded.

  By the assumption, take the leftmost path $f$ of $\bigoplus_k S_k$ and put $X = \{k : f(0,k) = 0\}$.
  It is easy to check that $X = \{k : [T_k] \neq \varnothing\}$.
\end{proof}

\subsection{Coded $\omega$-models}
We introduce the definition of coded $\omega$-models in $\RCAo$ and give some basic properties.

\begin{definition}
  ($\RCAo$)
  For a set $A$ and a sequence $ X = \langle X_i \rangle_i$, we say $A$ is in $\langle X_i \rangle_i$ or $\langle X_i \rangle_i$ contains $A$ (write $A \in \langle X_i \rangle _i$) if $\exists i(A = X_i)$, and thus
  $X$ is identified with an $\mathcal{L}_2$-structure $(\N,\langle X_i \rangle_i, 0,1,+,\cdot,\in)$.
  In this sense, we call $X$ a coded $\omega$-model.
\end{definition}

\begin{remark}
  Let $\M$ be a coded $\omega$-model.
  For a sentence $\sigma$ with parameters from $\M$, the satisfaction $\M \models \sigma$ is defined by
  the existence of an evaluating function for $\sigma$ which represents Tarski's truth definition and ensures $\sigma$ is true.
  If we work in $\ACAo$, then $\M \models \sigma$ if and only if
  the relativization $\sigma^{\M}$ is true.
  In particular, if $\sigma$ is arithmeical, then $\M \models \sigma$ if and only if $\sigma$ is true.
\end{remark}

 \begin{definition}\label{def-preliminaries-reflection}
   Let $\sigma$ be an $\L_2$-sentence.
   Define the assertion $\omega$-model reflection of $\sigma$ by
   \begin{equation*}
     \forall X \exists \M : \text{$\omega$-model } (X \in \M \land \M \models \sigma).
   \end{equation*}
 \end{definition}

 \begin{lemma}\label{reflection and con}
   Let $\sigma$ be a $\Pi^1_2$ sentence. Then, over $\ACAo$,
   the $\omega$-model reflection of $\sigma$ implies $\sigma + \Con(\sigma)$ where $\Con(\sigma)$ is the $\Pi^0_1$ formula
   stating that `$\sigma$ is consistent'.
   Therefore, if $\Con(\sigma)$ is not provable from $\ACAo$, then the $\omega$-model reflection of $\sigma$ is strictly stronger than $\sigma$.
 \end{lemma}
 \begin{proof}
   Write $\sigma \equiv \forall X \exists Y \theta(X,Y)$ by an arithmetical formula $\theta$.
   Assume $\ACAo$ and the $\omega$-model reflection of $\sigma$.
   We show that $\sigma + \Con(\sigma)$ holds.
   Let  $A$ be a set. By the $\omega$-model reflection, take an $\omega$-model $\M$ such that $A \in \M \land \M \models \forall X \exists Y \theta(X,Y)$.
   Then, $\M \models \exists Y \theta(A,Y)$ and hence $\theta(A,B)$ holds for some $B \in \M$.
   In addition, since $\M$ is a (weak-)model of $\sigma$, $\Con(\sigma)$ holds.
 \end{proof}

 The following characterization of $\ACA_0^+$ is well-known.
 \begin{theorem}\label{ACAo+ adn reflection}
   Over $\RCAo$, the following are equivalent.
   \begin{enumerate}
     \item $\ACA_0^+$,
     \item the $\omega$-model reflection of $\ACAo$; any set is contained in an $\omega$-model of $\ACAo$.
   \end{enumerate}
 \end{theorem}

\subsection{Hyperjump and $\beta$-model reflection}
In this subsection, we introduce the notion of hyperjump and $\beta$-models.
Hyperjump is a $\Sigma^1_1$ analogue of Turing jump.
As $\ACAo$ is characterized by Turing jumps, $\Pi^1_1\myhyphen\CAo$ is characterized by hyperjumps.

\begin{definition}\label{universal Sigma11}
  (Universal $\Sigma^1_1$ formula)
  Let $\Sat_{\Pi^0_0}$ be a universal $\Pi^0_0$ formula.
  We define the $\Sigma^1_1$ universal formula $\sigma^1_1$ by
  \begin{align*}
    \sigma^1_1(e,x,X) \equiv \exists f \forall y \Sat_{\Pi^0_0}(e,x,X,f[y]).
  \end{align*}.
\end{definition}

\begin{definition}
  ($\ACAo$)
  Let $\M$ be a coded $\omega$-model. We say $\M$ is a coded $\beta$-model if
  \begin{equation*}
    \forall X \in \M \forall e,x (\sigma^1_1(e,x,X) \leftrightarrow \M \models \sigma^1_1(e,x,X)).
  \end{equation*}
\end{definition}

Intuitively, a coded $\beta$-model is a $\Sigma^1_1$ absolute $\omega$-model. In fact, we can show the following.

\begin{lemma}
  Let $\theta(x,X)$ be a $\Sigma^1_1$ formula.
  Then, $\ACAo$ proves that for any $\beta$-model $\M$, $\forall X \in \M \forall x (\theta(x,X) \leftrightarrow \M \models \theta(x,X))$.
\end{lemma}
\begin{proof}
Immediate from the definition of $\beta$-models.
\end{proof}

We next introduce the notion of hyperjump.
Intuitively, the hyperjump operator is an operator to make a $\Sigma^{1,X}_1$ complete set from a set $X$.
\begin{definition}
  ($\ACAo$)
  Let $X$ be a set. Define the hyperjump $\HJ(X)$ by $\HJ(X)  = \{(e,x) : \sigma^1_1(e,x,X)\}$.
  We write $\HJ^n(X)$ for the $n$-times iteration of the hyperjump.
 \end{definition}

\begin{lemma}
  Let $\varphi(x,y,X)$ be a $\Sigma^1_1$ formula with exactly displayed free variables.
  Then, $\ACAo$ proves
  \begin{equation*}
    \forall y,X,Y (Y = \HJ(X) \to \exists Z \leq_{\T} Y (Z = \{x : \varphi(x,y,X)\}) ).
  \end{equation*}
  That is, any $\Sigma^{1,X}_1$-definable set is computable from $\HJ(X)$.
  Hence, any $\Pi^{1,X}_1$-definable set is also computable from $\HJ(X)$.
\end{lemma}
\begin{proof}
Immediate from the definition of $\HJ$.
\end{proof}

\begin{lemma}[\cite{Simpson} Lemma VII.2.9]
  Over $\ACAo$, for any set $X$, the existence of $\HJ(X)$ is equivalent to the existence of a coded $\beta$-model containing $X$.
\end{lemma}

Combining the above two lemmas, we have that
\begin{theorem}\label{equivalence of HJ and Beta}
  The following assertions are equivalent over $\ACAo$.
  \begin{enumerate}
    \item $\Pi^1_1\myhyphen\CAo$,
    \item $\forall X \exists Y (Y = \HJ(X))$,
    \item $\beta$-model reflection : $\forall X \exists \M \text{ : $\beta$-model } (X \in \M)$.
  \end{enumerate}
\end{theorem}
We note that all of these assertions are $\Pi^1_3$ statements and any of them is never equivalent to a $\Pi^1_2$ statement.
In this paper, we introduce $\Pi^1_2$ variants of the second and third assertions to consider the structure of $\Pi^1_2$ consequences of $\Pi^1_1\myhyphen\CAo$.

At last, we see that a hyperjump in the ground model behaves a hyperjump in a coded $\omega$-model.
The following two lemmas are refinements of \cite[VII.1.12]{Simpson}.

\begin{lemma}\label{absoluteness of hyperjump}
  Let $\M = (\N^{\M},S^{M})$ and $\M' = (\N^{\M'},S^{\M'})$ be models of $\ACAo$ such that
  $\N^{\M} = \N^{\M'}$ and $S^{\M'} \subseteq S^{\M}$.

  Let $A,B \in \M'$ such that $\M \models B= \HJ(A)$. Then, $\M' \models B = \HJ(A)$.
  Moreover, if $\M'$ is a coded $\omega$-model in $\M$, then $\M \models (B = \HJ(A))^{\M'}$.
\end{lemma}
\begin{proof}
  Take $\M,\M',A$ and $B$ be as above.
  We first show the following claim.
  \setcounter{claimcounter}{1} 
  \begin{claim}
    Let $\theta(X)$ be a $\Sigma^1_1$ formula having exactly one set variable.
    Then, $\M \models \theta(A)$ if and only if $\M' \models \theta(A)$.
  \end{claim}
  \addtocounter{claimcounter}{1}
  \begin{proof of claim}
    It is enough to show that $\M \models \theta(A)$ implies $\M' \models \theta(A)$.
    Write $\theta(X) \equiv \exists Z \varphi(X,Z)$ by an arithmetical formula $\varphi$.
    Assume $\M \models \exists Z \varphi(A,Z)$. We will show that $\M' \models \exists Z \leq_{\T} B \varphi(A,B)$.
    Now there is a $\Sigma^0_0$ formula $\theta$ such that both of $\M,\M'$ satisfies
    $\forall Z (\varphi(A,Z) \leftrightarrow \exists f \forall n \theta(A,Z[n],f[n]))$.
    Define trees $T,T' \subseteq (2 \times \N)^{<\N}$ by
    \begin{align*}
      T &= \{(s,\sigma) : \forall (t,\tau) \preceq (s,\sigma) \theta(A,t,\tau)\}, \\
      T' &= \{(s,\sigma) : \exists (Z,f) \in [T]  ((s,\sigma) \prec (Z,f)) \}.
    \end{align*}
    Then, $T$ is $\Delta^0_1$ in $A$ and hence $T'$ is $\Sigma^1_1$ in $A$.
    Therefore, $T' \leq_{\T} \HJ(A)$.
    Moreover, since $\exists Z \varphi(X,Z)$ holds, $T'$ is a nonempty pruned tree.
    Thus $T'$ has a $T'$-computable path $(Z,f)$. Now $Z \leq_{\T} T' \leq_{\T} \HJ(X)$, so this completes the proof.
    \qedclaim
  \end{proof of claim}
  Now $B = \{(e,x) \in \N^{\M} :  \M \models \sigma^1_1(e,x,A)\} = \{(e,x) \in \N^{\M'} : \M' \models \sigma^1_1(e,x,A)\}$.
  Therefore $\M' \models B = \HJ(A)$.
\end{proof}

\begin{lemma}\label{lemma closed under hyperjump then beta-submodel of pi11ca}
  Let $\M = (\N^{\M},S^{M})$ and $\M' = (\N^{\M'},S^{\M'})$ be such that $\M$ is a model of $\ACAo$,
  $\N^{\M} = \N^{\M'}$, $S^{\M'} \subseteq S^{\M}$ and $S^{\M'}$ is closed under Turing sum $\oplus$ and lower closed under Turing reduction.
  If for any $A \in \M'$, there exists $B \in \M'$ such that
  $\M \models B = \HJ(A)$, then $\M'$ is a model of $\Pi^1_1\myhyphen\CAo$ and a $\beta$-submodel of $\M$.
\end{lemma}
\begin{proof}
  Since $S^{\M'}$ is closed under Turing sum and lower closed under Turing reduction, $\M'$ is a model of $\RCAo$.
  Moreover, by the definition of hyperjumps, it is easy to see that $\M'$ is closed under $\Pi^0_1$-comprehension.
  Therefore, $\M'$ is a model of $\ACAo$.

  By the previous lemma, $\M' \models \forall X \exists Y (Y = \HJ(X))$. Therefore, $\M'$ is a model of $\Pi^1_1\myhyphen\CAo$.

  To show $\M'$ is a $\beta$-submodel of $\M$, it is enough to show that for any $\Sigma^1_1$ sentence with parameters from $\M'$, if it is true in $\M$, then it is also true in $\M'$.
  Let $\theta(X,Y)$ be an arithmetical formula and $A \in \M'$.
  Assume $\M \models \exists Y \theta(A,Y)$. We show that
  $\M' \models \exists Y \theta(A,Y)$.
  Let $B \in \M$ such that $\M \models B = \HJ(A)$.
  Then, $\M \models \exists Y \leq_{\T} B \theta(A,Y)$ and $B \in \M'$.
  Thus, $\M' \models \exists Y \leq_{\T} B \theta(A,Y)$.
\end{proof}

\begin{lemma}\label{Hyperjump in the ground model is also hyperjump in a coded model}
  ($\ACAo$)
  Let $X,Y$ be sets such that $Y = \HJ(X)$.
  Then, for any coded $\omega$-model $\M$ of $\ACAo$, if $Y \in \M$, then $X \in \M \land \M \models Y = \HJ(X)$.
\end{lemma}
\begin{proof}
  We reason in $\ACAo$. Take a pair of sets $X,Y$ such that $Y = \HJ(X)$.
  Let $\M$ be a coded $\omega$-model such that $Y \in \M$.

  By Lemma \ref{absoluteness of hyperjump}, we only need to show $X \in \M$.
  Take a number $e$ such that
  \begin{equation*}
    \forall Z (x \in Z \leftrightarrow \sigma^1_1(e,x,Z)).
  \end{equation*}
  Then, $X = \{x : (e,x) \in Y\}$. Since $X \leq_{\T} Y$, $X \in \M$.
\end{proof}

\begin{remark}
  In general, even if $(Y = \HJ(X))^{\M}$ holds for a coded $\omega$-model $\M$, $Y = \HJ(X)$ may not hold.
  This is immediate from some results  in the next section.
\end{remark}

\section{Approximation of $\Pi^1_1\myhyphen\CAo$}
In this section, we introduce $\Pi^1_2$ variants of the $\beta$-model reflection $\beta^1_0\RFN(n)$.
Then we prove that  the theories $\ACAo + \{\beta^1_0\RFN(n): n \in \omega\}$ and
$\{\sigma \in \Pi^1_2 : \Pi^1_1\myhyphen\CAo \vdash \sigma\}$ proves the same sentences.

As we have seen in Theorem \ref{equivalence of HJ and Beta},  $\Pi^1_1\myhyphen\CAo$ is characterized by the hyperjump operator or the $\beta$-model reflection.
This means that in $\Pi^1_1\myhyphen\CAo$, we can take any (standard) finite iteration of hyperjump or $\beta$-models. More precisely, the following holds.
\begin{observation}
  For any $n \in \omega, n >0$, the following assertions are equivalent over $\ACAo$.
  \begin{itemize}
    \item $\Pi^1_1\myhyphen\CAo$,
    \item $\forall X \exists Y (Y = \HJ^n(X))$,
    \item $\forall X \exists \M_0,\ldots,\M_{n-1} :\text{$\beta$-models } (X \in \M_0 \in \cdots \in \M_{n-1})$.
  \end{itemize}
\end{observation}
From this observation, we can conjecture that theorems provable from $\Pi^1_1\myhyphen\CAo$ are classified by
the number of the hyperjump operator used to proved those theorems.
The following theorem ensures this conjecture.

\begin{theorem}\label{compactness for hyperjump}
  Let $\theta(X,Y,Z)$ be an arithmetical formula such that
  $\Pi^1_1\myhyphen\CAo \vdash \forall X \exists Y \forall Z \theta(X,Y,Z)$.
  Then, there exists $n \in \omega$ such that
  \begin{equation*}
    \ACAo \vdash \forall X,W( W = \HJ^n(X) \to \exists Y \leq_{\T} W \forall Z \theta(X,Y,Z))).
  \end{equation*}
\end{theorem}
\begin{proof}
  Assume $\theta(X,Y,Z)$ be an arithmetical formula such that
  $\Pi^1_1\myhyphen\CAo \vdash \forall X \exists Y \forall Z \theta(X,Y,Z)$.
  For the sake of contradiction, suppose that for any $n \in \omega$,
  $\ACAo + \exists X,W (W = \HJ^n(X) \land \forall Y \leq_{\T} W \exists Z \lnot \theta(X,Y,Z))$ is consistent.

  Let $C$ and $D$ be new constant set symbols.
  Let $\mathcal{L}_2(C,D)$ be the language extending $\mathcal{L}_2$ by adding $C,D$. Define an $\mathcal{L}_2(C,D)$ theory $T$ by
  \begin{align*}
    T = \ACAo &+ \{D_0 = C \land D_1 = \HJ(D_0) \land \cdots \land D_k = \HJ(D_{k-1}) : k \in \omega\} \\
    &+ \{\forall Y \leq_{\T} D_k \exists Z \lnot \theta(C,Y,Z) : k \in \omega\}.
  \end{align*}
  Then, by compactness theorem, $T$ has a model $\M = (\N^{\M},\widetilde{S},C,D)$.
  Define $S = \bigcup_{n \in \omega}\{X \in \widetilde{S} : \M \models X \leq_{\T} D_n\}$.

  By the definition, for any $X \in S$, there exists $Y \in S$ such that
  $\M \models Y = \HJ(X)$.
  Thus, by Lemma \ref{lemma closed under hyperjump then beta-submodel of pi11ca}, $(\N,S)$ is a model of $\Pi^1_1\myhyphen\CAo$ and a $\beta$-submodel of $(\N,\widetilde{S})$.
  We show that $(\N^{\M},S)$ is a model of  $\lnot \forall X \exists Y \forall Z \theta(X,Y,Z)$.
  In particular, $(\N^{\M},S) \models \forall Y \exists Z \lnot \theta(C,Y,Z)$.
  Let $Y \in S$. Then, $\M \models Y \leq_{\T} D_k$ for some $k$ and thus $\M \models \exists Z \lnot \theta(C,Y,Z)$. Hence $(\N,S) \models \exists Z \lnot \theta(C,Y,Z)$.
\end{proof}

\begin{remark}
  We note that this theorem generalizes a result by Montalban and Shore \cite[Theorem 6.8.]{montalban2018conservativity}.
  They proved that if
  \begin{align*}
    \Pi^1_1\myhyphen\CAo \vdash \forall X \exists Y \forall Z \theta(X,Y,Z),
  \end{align*}
  then there exists $n \in \omega$ such that
  \begin{equation*}
    \Pi^1_1\myhyphen\CAo \vdash \forall X  \exists Y \leq_{\T} \HJ^n(X) \forall Z \theta(X,Y,Z))).
  \end{equation*}
\end{remark}

We give a $\Pi^1_2$ variant of
the statement that says \textit{any set is contained in a coded $\beta$-model}.
For this purpose,
we consider $\beta$-submodels of a coded $\omega$-model of $\ACAo$ instead of the acutual $\beta$-models.

\begin{definition}
  ($\ACAo$)
Let $\M_0$ and $\M_1$ be coded $\omega$-models such that $\M_0 \in \M_1$.
We say $\M_0$ is a coded $\beta$-submodel of $\M_1$ if
\begin{align*}
  \forall X \in \M_0 \forall e,x (\M_0 \models \sigma^1_1(e,x,X) \leftrightarrow \M_1\models \sigma^1_1(e,x,X))
\end{align*}
holds. We write $\M_0 \in_{\beta} \M_1$ to mean $\M_0$ is a coded $\beta$-submodel of $\M_1$.
\end{definition}

\begin{definition}
  For each $n \in \omega$, define $\beta^1_0\RFN(n)$ as the following assertion:
  \begin{equation*}
    \forall X \exists \M_0,\ldots,\M_{n}: \text{coded $\omega$-models }
    (X \in \M_0 \in_\beta \cdots \in_\beta \M_n \land \M_n \models \ACAo).
  \end{equation*}
\end{definition}
\begin{remark}
  This form of $\omega$-model reflection is first introduced in \cite{pacheco2022determinacy} as their $\psi_0(1,n)$.
  They mainly considered any (nonstandard) length of sequences of $\beta$-submodels, but we mainly consider standard length of those.
\end{remark}

It is easy to see that the implications $\beta^1_0\RFN(0) \leftarrow \beta^1_0\RFN(1) \leftarrow \cdots $ holds over $\ACAo$.
More precisely, these implications are strict.
\begin{lemma}\label{beta(n+1) is strictly stronger than beta(n)}
  Let $n \in \omega$ such that $n > 0$. Then, over $\ACAo$,
  $\beta^1_0\RFN(n)$ proves the $\omega$-model reflection of $\beta^1_0\RFN(n-1)$.
  In particular, $\beta^1_0\RFN(n)$ is strictly stronger than  $\beta^1_0\RFN(n-1)$.
\end{lemma}
\begin{proof}
  Assume $\ACAo + \beta^1_0\RFN(n+1)$.
  Take a sequence of $\omega$-models such that
  $\M_0 \in_\beta \M_1 \cdots \in_\beta \M_{n}  \land \M_n \models \ACAo$.
  We show that $\M_0 \models \beta^1_0\RFN(n)$.
  For any $X \in \M_0$, $\M_{n}$ satisfies
  \begin{align*}
    \exists \mathcal{N}_0,\cdots,\mathcal{N}_{n-1}
    (X \in \mathcal{N}_0 \in_\beta \cdots \in_\beta \mathcal{N}_{n-1} \land \mathcal{N}_{n-1} \models \ACAo)
  \end{align*}
  via taking $(\mathcal{N}_0,\ldots,\mathcal{N}_{n-1}) = (\M_0,\ldots,\M_{n-1})$.
  Since this is a $\Sigma^1_1$ formula and $\M_0$ is a $\beta$-submodel of $\M_{n}$, $\M_0$ also satisfies
  \begin{align*}
    \exists \mathcal{N}_0,\cdots,\mathcal{N}_{n-1}
    (X \in \mathcal{N}_0 \in_\beta \cdots \in_\beta \mathcal{N}_{n-1} \land \mathcal{N}_{n-1} \models \ACAo).
  \end{align*}
  Since $X$ is arbitrary, $\M_0 \models \beta^1_0\RFN(n)$.
\end{proof}

\begin{lemma}
  Over $\ACAo$, $\beta^1_0\RFN(n)$ is equivalent to the following assertion:
  \begin{equation*}
    \forall X \exists \M : \text{coded $\omega$-model }
    (X \in \M \land \M \models \ACAo + \exists Y(Y = \HJ^n(X))).
  \end{equation*}
\end{lemma}
\begin{proof}
  Immediate from the fact the existence of $\HJ(X)$ is equivalent to the existence of a $\beta$-model containing $X$ over $\ACAo$.
\end{proof}
\begin{theorem}\label{Pi12part and hyperjumps}
  Let $\theta(X,Y)$ be an arithmetical formula such that $\forall X \exists Y \theta(X,Y)$ is provable from $\Pi^1_1\myhyphen\CAo$.
  Then, there exists $n \in \omega$ such that
  $\ACAo + \beta^1_0\RFN(n)$ proves $\forall X \exists Y \theta(X,Y)$.
\end{theorem}
\begin{proof}
  Let $\sigma_n$ be the following sentence.
  \begin{align*}
    \forall X \exists \M : \text{coded $\omega$-model }(X \in \M \land \M \models \ACAo + \exists Y(Y = \HJ^n(X))).
  \end{align*}
  Then, $\beta^1_0\RFN(n)$ is equivalent to $\sigma_n$ over $\ACAo$.

  We show that $\forall X \exists Y \theta(X,Y)$ is provable from
  the theory $\ACAo + \{\sigma_n : n \in \omega\}$.
  Since $\forall X \exists Y \theta(X,Y)$ is provable from $\Pi^1_1\myhyphen\CAo$,
  there exists $n \in \omega$ such that
  \begin{equation*}
    \ACAo \vdash \forall X,W (W= \HJ^n(X) \to \exists Y \leq_{\T} W \theta(X,Y))
  \end{equation*}
  by Thereom \ref{compactness for hyperjump}.
  Take such an $n$. Assume $\ACAo + \sigma_n$. Let $X$ be an arbitrary set. Then, there is a coded $\omega$-model $\M$ such that
  \begin{equation*}
    X \in \M \land \M \models \ACAo + \exists W(W = \HJ^n(X)).
  \end{equation*}
  Take such $\M$ and $W$.
  Then $\M \models \exists Y \leq_{\T} W \theta(X,Y)$ and hence $\M \models \exists Y \theta(X,Y)$.
  Since $\theta$ is arithmetical, we have $\exists Y \theta(X,Y)$.
  This completes the proof.
\end{proof}
Since each $\beta^1_0\RFN(n)$ is a $\Pi^1_2$ sentence provable from $\Pi^1_1\myhyphen\CAo$,
the theory $\ACAo + \{\beta^1_0\RFN(n) : n \in\omega\}$ identifies the family of all $\Pi^{1}_{2}$-consequences of $\Pi^1_1\myhyphen\CAo$, $\{\sigma \in \Pi^1_2: \Pi^1_1\myhyphen\CAo \vdash \sigma\}$.
In Sections 4--7, we see the relationship between the $\beta^1_0\RFN$-hierarchy and certain $\Pi^1_2$ theorems provable from $\Pi^1_1\myhyphen\CAo$.
We note that $\Pi^1_1\myhyphen\CAo$ may prove $\Pi^{1}_{2}$ theorems faster than $\ACAo + \{\beta^1_0\RFN(n) : n \in\omega\}$.
More precisely, $\Pi^1_1\myhyphen\CAo$ has iterated exponential speedup over $\ACAo + \{\beta^1_0\RFN(n) : n \in\omega\}$ when $\beta^1_0\RFN(n)$ is expressed by using the $n$-th numeral.
This can be seen as follows. $\Pi^1_1\myhyphen\CAo$ proves that $\beta^1_0\RFN(x)$ is inductive (\textit{i.e.}, $\forall x(\beta^1_0\RFN(x)\to \beta^1_0\RFN(x+1))$ holds). Then, by Solovay's shortening method (see 3.4--3.5 of \cite{incoll:pflen}), we get a poly$(n)$-size proof of $\beta^{1}_{0}\RFN(2_{n})$ (where $2_{0}=1$ and $2_{n+1}=2^{2_{n}}$, and note that $2_{x}$ can be expressed as a $\Sigma_{1}$ formula of size $O(\log n)$). Since $\beta^{1}_{0}\RFN(2_{n})$ implies the consistency of $\ACAo+\beta^{1}_{0}\RFN(m)$ for all $m<2_{n}$, the proof of $\beta^{1}_{0}\RFN(2_{n})$ from $\ACAo + \{\beta^1_0\RFN(n) : n \in\omega\}$ exactly needs the $2_{n}$-th axiom of $\beta^1_0\RFN$ whose size is bigger than $2_{n}$.
\nocite{book:hb-pfthy}

In the remaining of this section, we see some basic properties of $\beta^1_0\RFN(n)$.
We note that $\beta^1_0\RFN(0)$ is the same as the $\omega$-model reflection of $\ACAo$.
Thus, $\beta^1_0\RFN(0)$ is equivalent to $\ACA_0^+$ over $\RCAo$.
Therefore, the $\beta^1_0\RFN$-hierarchy is above $\ACA_0^+$. Moreover, it is easy to prove that even $\beta^1_0\RFN(1)$ is stronger than $\ATRo$.

\begin{lemma}[\cite{Simpson}, Theorem VII.2.7.]\label{beta-model satisfies atro}
  Over $\ACAo$, any coded $\beta$-model satisfies $\ATRo$.
\end{lemma}

\begin{theorem}
  Over $\ACAo$, $\beta^1_0\RFN(1)$ implies the $\omega$-model reflection of $\ATRo$.
\end{theorem}
\begin{proof}
  We reason in $\ACAo + \beta^1_0\RFN(1)$.
  Let $X$ be a set. By $\beta^1_0\RFN(1)$, take coded $\omega$-models $\M_0$ and $\M_1$ such that
  \begin{align*}
    X \in \M_0 \in_{\beta} \M_1 \land \M_1 \models \ACAo.
  \end{align*}
  Then, $\M_1 \models [\M_0 \models \ATRo]$.
  Thus
  $\M_0 \models \ATRo$ actually holds.
\end{proof}

In the following section, we will show that there is a very large gap between $\ATRo$ and $\beta^1_0\RFN(1)$.

\section{Relativized leftmost path principle and $\beta^1_0\RFN$}
As we have noted in introduction, Towsner introduced a family of $\Pi^1_2$ statements called relative leftmost path principles to give an upper bound of some $\Pi^1_2$ theorems provable from $\Pi^1_1\myhyphen\CAo$ \cite{Townser_TLPP}.
Indeed, he proved that Kruskal's theorem for trees,  Nash-Williams' theorem in bqo theory and Menger's theorem for countable graphs are provable from one of the relative leftmost path principles.

In this section, we compare the relative leftmost path principles and the $\beta^1_0\RFN$-hierarchy.
More precisely, we introduce a variant of relative leftmost path principle which we call arithmetical relative leftmost path principle ($\ALPP$) and show that $\ALPP$ is equivalent to $\beta^1_0\RFN(1)$,
and we show that the strongest form of relative leftmost path principle, the transfinite relativized leftmost path principle $\TLPP$, is weaker than $\beta^1_0\RFN(2)$.
Moreover, we introduce the $n$-th iteration of $\ALPP$ denoted $\It^n(\ALPP)$, which is equivalent to $\beta^1_0\RFN(n)$.

To define $\ALPP$, we introduce the notion of arithmetical reducibility.
\begin{definition}[Arithmetical reducibility, $\RCAo$]
  Let $X,Y$ be sets.
  We say $X$ is arithmetically reducible to $Y$ (write $X \leq_{\T}^{\a} Y$) if
  \begin{equation*}
    \exists n \exists Z (Z = Y^{(n)} \land X \leq_{\T} Z).
  \end{equation*}

  Let $\varphi(X)$ be a formula.
  We write $\forall X \leq_{\T}^{\a} Y \varphi(X)$ to mean $\forall X (X \leq_{\T}^{\a} Y \to \varphi(X))$ and write $\exists X \leq_{\T}^{\a} Y \varphi(X)$ to mean $\exists X ( X \leq_{\T}^{\a} Y \land \varphi(X))$.
\end{definition}

\begin{remark}
The notation $\leq_{\T}^{\a}$ follows from the notation $\leq_{\mathrm{W}}^{\a}$ of arithmetical Weihrauch reducibility.
\end{remark}

\begin{remark}
  We note that arithmetical reducibility is an extension of the notion of arithmetical definability.
  For $A,B \subset \omega$, we say $A$ is arithmetically definable from $B$ if there is an arithmetical formula
  $\theta(x,X)$ with exactly displayed free variables such that $A = \{x : \theta(x,B)\}$.
  Thus, $A$ is arithmetically definable from $B$ is equivalent to the condition that for any $\omega$-model $\M$ of $\ACA_0$ containing $A$ and $B$,
  $\M \models A \leq^{\a}_{\T} B$.
\end{remark}
In contrast to this remark, for a non $\omega$-model $\M$ of $\ACAo$ and $A,B \in \M$,
$\M \models A \leq^{\a}_{\T} B$ does not mean that there is an arithmetical formula $\theta$ such that
$\M \models A = \{x: \theta(x,B)\}$.
However, the connection between arithmetical reducibility and arithmetical definability still holds in the following sense.
\begin{lemma}
  $(\ACAo)$ Let $\M$ be a coded $\omega$-model of $\ACAo$.
  Then, $\M$ is closed under arithmetical reduction. Hence, for any arithmetical formula $\theta$ and $Y \in \M$,
  \begin{align*}
    \exists X \leq^{\a}_{\T} Y \theta(X) &\to [\M \models \exists X \leq^{\a}_{\T} Y \theta(X)] \text{ and } \\
    [\M \models \forall X \leq^{\a}_{\T} Y \theta(X)] &\to \forall X \leq^{\a}_{\T} Y \theta(X).
  \end{align*}
\end{lemma}
\begin{proof}
  Let $\M$ be an $\omega$-model of $\ACAo$.
  Then, $\M$ satisfies $\Sigma^1_1$ induction and $\forall Y \exists Z (Z = \mathrm{TJ}(Y))$ where 
  $\mathrm{TJ}$ denotes the Turing jump operator.
  Thus $\M \models \forall Y \forall n \exists Z(Z = Y^{(n)})$.

  To show that $\M$ is closed under arithmetical reduction, take $Y \in \M$ and $X$ such that $X \leq_{\T}^{\a} Y$.
  Then, there exist $n,Z$ such that $Z = Y^{(n)} \land X \leq_{\T} Z$.
  We note that $Z \in \M$ because the condition $Z = Y^{(n)}$ is absolute for $\M$.
  Thus, $X \in \M$.

  The latter part immediately follows from the former part.
\end{proof}

\begin{lemma}\label{arithmetical red is delta11}
  Let $\varphi(X)$ be an arithmetical formula.
  Then, over $\mathsf{ACA}_0^+$, $\forall X \leq_{\T}^{\a} Y \varphi(X)$ is $\Delta^1_1$.
  Similarly, $\exists X \leq_{\T}^{\a} Y \varphi(X)$ is $\Delta^1_1$.
\end{lemma}
\begin{proof}
  Within $\ACA_0^+$, $\forall X \leq_{\T}^{\a} Y \varphi(X)$ can be written as follows.
  \begin{align*}
    \forall X \leq_{\T}^{\a} Y \varphi(X) &\leftrightarrow \forall Z(Z = Y^{(\omega)} \to \forall n \forall X \leq_{\T} Z_n \varphi(X)) \\
    &\leftrightarrow \exists Z(Z = Y^{(\omega)} \land \forall n \forall X \leq_{\T} Z_n \varphi(X))
  \end{align*}
  where $Z_n$ denotes the $n$-th segment of $Z$ as introduced in Section 2. 
  Thus, in this case, $Z_n$ is $Y^{(n)}$.
  From the above equuivalences, $\forall X \leq_{\T}^{\a} Y \varphi(X)$ is $\Delta^1_1$.
  The case for $\exists X \leq_{\T}^{\a} Y \varphi(X)$ is the same.
\end{proof}
\begin{lemma}[Folklore]\label{arithmetical reduction and omega-model}
  Assume $\ACA_0^{+}$. Then, for any $X$, the coded model $\M = \{Y : Y \leq_{\T}^{\a} X\}$ exists.
  Moreover,
  for any coded $\omega$-model $\M'$ of $\ACAo$ containing $X$,
  \begin{align*}
    \forall Y (Y \in \M \to Y \in \M')
  \end{align*}
  holds. In this sense, $\M$ is the smallest coded $\omega$-model of $\ACAo$ containing $X$.
\end{lemma}

\begin{definition}[Relative leftmost path principle, See \cite{Townser_TLPP}]
  Let $k \in \omega$. We define $\Delta^0_k\LPP$ as the assertion that
  for any ill-founded tree $T$, there is a path $g \in [T]$ such that
  \begin{equation*}
    \forall f \in \Delta^0_k(T \oplus g) (f \in [T] \to g \leq_l f).
  \end{equation*}

  We define $\ALPP$ as the assertion that for any ill-founded tree $T$,
  there is a path $g \in [T]$ such that
  \begin{equation*}
    \forall f \leq^{\a}_{\T} (T\oplus g) (f \in [T] \to g \leq_l f).
  \end{equation*}
  We call a witness $g$ of $\ALPP$ an arithmetical leftmost path.

  Finally, over $\ATRo$, define $\TLPP$ as the following assertion:
  For any ill-founded tree $T$ and any well-order $\alpha$, there is a path $g \in [T]$ such that
  \begin{equation*}
    \forall f \leq_{\T} (T \oplus g)^{(\alpha)} (f \in [T] \to g \leq_l f).
  \end{equation*}
  We call such $g$ a $\Delta^0_{\alpha+1}$ leftmost path.
  \end{definition}

  \begin{remark}
    In the definition of $\Delta^0_k\LPP,\ALPP$ and $\TLPP$, we can replace $(T \oplus g)$ by $g$ if we work in $\RCAo$.
  \end{remark}

From now on, we see the relationship between $\beta^1_0\RFN(1),\beta^1_0\RFN(2)$ and $\ALPP,\TLPP$.
For this comparison, we prefer the base theory to be $\ACA_0^+$ rather than $\ACAo$ to simplify the discussion by Lemma \ref{arithmetical red is delta11} and \ref{arithmetical reduction and omega-model}.
However, the following lemma take the difference of the choice of the base theory away.
\begin{lemma}
  Over $\ACAo$, $\Delta^0_0\LPP$ implies $\ATRo$.
\end{lemma}
\begin{proof}
  See \cite{Townser_TLPP}, Theorem 4.2.
\end{proof}

Our strategy to prove $\beta^1_0\RFN(1)$ from $\ALPP$ is essentially the same as for the proof of Theorem \ref{LPP and Pi11CA};
to make a hyperjump, it is sufficient
to make the set $\{i : [T_i] \neq \varnothing\}$ from given sequence of trees $\langle T_i \rangle_i$.

\begin{lemma}
  There are a $\Sigma^0_1$ formula $\varphi(n,X)$ and a $\Pi^0_1$ formula $\psi(n,X)$ such that
  \begin{itemize}
    \item Over $\RCAo$, $\forall n,X (\varphi(n,X) \leftrightarrow \psi(n,X))$.
    \item Over $\RCAo$, for any $X$, $T = \{n: \varphi(n,X)\}$ is a tree such that $1^\infty \in [T]$.
    \item Over $\ACAo$, for any $X$, if $\{n: \varphi(n,X)\}$ has a leftmost path, then $\HJ(X)$ is computable from it. Moreover, this reduction is uniform.
  \end{itemize}
  We may assume that there is a Turing functional $\Phi_T(X)$ defined by $\Phi_T(X) = \{n: \varphi(n,X)\}$.
\end{lemma}
\begin{proof}
  Recall that there is a $\Sigma^0_0$ formula $\theta$ such that over $\ACAo$,
  $\HJ(X) = \{(e,x) : \exists f \forall z \theta(e,x,X,f[z])\}$ for any $X$.

  We give a construction of $\Phi_T(X)$ over $\RCAo$.
  For given $X$, define a sequence $\langle T_{e,x} \rangle_{e,x}$ of trees by
  $\sigma \in T_{e,x} \leftrightarrow \theta(e,x,X,\sigma)$.
  Then put $S_{e,x} = \{ \langle 0 \rangle \ast \sigma : \sigma \in T_{e,x}\} \cup \{\sigma : \forall n < |\sigma| \sigma(n) = 1\}$.
  Finally, put $\Phi_T(X) = \bigoplus_{e,x} S_{e,x}$ where $\bigoplus$ is the operator defined in Definition \ref{oplus of trees}.

  We show that this construction satisfies desired conditions.
  It is easy to see that $\Phi_T(X)$ is uniformly computable from $X$ over $\RCAo$ and its has an infinite path $1^{\infty}$.
  Thus, it is enough to show that $\HJ(X)$ is uniformly computable from the leftmost path of $\Phi_T(X)$ over $\ACAo$.

  We reason in $\ACAo$. Let $f$ be the leftmost path of $\Phi_T(X)$.
  Then, for each $e$ and $x$, $f((0,(e,x))) = 0$ iff $T_{e,x}$ has a path.
  Thus, $\{(e,x) : f((0,(e,x))) = 0 \} = \{(e,x) : \exists f \forall z \theta(e,x,X,f[z],z)\} = \HJ(X)$.
\end{proof}

  We next see that we can modify the second condition as follows :
  \begin{itemize}
    \item Over $\RCAo$, for any $X$, $\Phi_T(X) = \{n : \varphi(n,X)\}$ is an ill-founded tree which has a path uniformly computable from $X$, and any path computes $X$.
  \end{itemize}
  For this condition, consider the following transformation.
  \begin{definition}
    Let $\sigma \in \N^{<\N}$. Put
    \begin{align*}
      &\sigma_{\even} = \langle \sigma(0),\sigma(2),\ldots,\sigma(l_e)\rangle, \\
      &\sigma_{\odd} = \langle \sigma(1),\sigma(3),\ldots,\sigma(l_o)\rangle,
    \end{align*}
    where $l_e$ is the maximum even number such that $l_e < |\sigma|$ and $l_o$ is the maximum odd number such that $l_o < |\sigma|$.
    Similarly, for a function $f \in \N^{\N}$, define $f_{\even} = \langle f(2n) \rangle_{n \in \N}$ and $f_{\odd} = \langle f(2n+1) \rangle_{n \in \N}$.

    Let $X$ be a set and $T$ be a tree. We identify $X$ and its characteristic function. Thus, $X$ is regarded as an infinite binary sequence. Define a tree $S_{T,X}$ by
    \begin{equation*}
      \sigma \in S_{T,X} \leftrightarrow \sigma_{\even} \in T \land \sigma_{\odd} \prec X.
    \end{equation*}
  \end{definition}
  \begin{lemma}
    Let $X$ be a set and $T$ be a tree. Then, over $\RCAo$,
    \begin{itemize}
      \item $S_{T,X}$ is uniformly computable from $X$ and $T$.
      \item For any path $f$ of $S_{T,X}$, $f_{\even}$ is a path of $T$ and $f_{\odd} = X$. Conversely, for $f \in [T]$ and $X$, the sequence $f \uplus X = \langle f(0),X(0),f(1),X(1),\ldots \rangle$ is a path of $S_{T,X}$.
      \item If $f$ and $g$ are paths of $T$ such that $f \leq_l g$, then $f \uplus X \leq_l g \uplus X$. In particular, if $f$ is the leftmost path of $T$, then the corresponding path $f \uplus X$ is the leftmost path of $S_{T,X}$.
    \end{itemize}
  \end{lemma}
  \begin{proof}
    This is immediate from the definition of $S_{T,X}$.
  \end{proof}
  Therefore, by replacing $\Phi_T(X)$ with $S_{\Phi_T(X),X}$, we have
  \begin{lemma}\label{lem-of-Phi_T}
    There is a Turing functional $\Phi_T(\bullet)$ defined in $\RCAo$ such that
      \begin{itemize}
        \item Over $\RCAo$, for any $X$, $\Phi_T(X)$ is an ill-founded tree which has a path uniformly computable from $X$, and any path computes $X$.
        \item Over $\ACAo$, for any $X$, if $\Phi_T(X)$ has a leftmost path, then $\HJ(X)$ is computable from that path. Moreover, this reduction is uniform.
      \end{itemize}
  \end{lemma}

\begin{theorem}
  Over $\ACAo$, $\ALPP$ is equivalent to $\beta^1_0\RFN(1)$.
\end{theorem}
\begin{proof}
  First, we show that $\ALPP$ implies $\beta^1_0\RFN(1)$.
  Since $\Delta^0_0\LPP$ implies $\ATRo$, it is enough to show that
  $\ACA_0^+ + \ALPP$ proves $\beta^1_0\RFN(1)$.

  Let $X$ be a set. By $\ALPP$, take an arithmetical leftmost path $g$ of $\Phi_T(X)$.
  Let $\M = \{Y : Y \leq_{\T}^{\a} g \}$.
  Then $X,\Phi_T(X) \in \M$, $\M \models \ACAo$ and
  $g$ is the leftmost$^\M$ path of $\Phi_T(X)$.
  Thus, $\M$ satisfies $\HJ(X)$ exists.

  We next show that $\beta^1_0\RFN(1)$ implies $\ALPP$.
  Let $T$ be an ill-founded tree.
  By $\beta^1_0\RFN(1)$, take an $\omega$-model $\M$ such that
  $T \in \M$ and $\M \models \ACAo + \exists Y(Y = \HJ(T))$.

  Now, $\M$ satisfies [$T$ has a leftmost path] because its hyperjump exists. Take the leftmost$^\M$ path $g$ of $T$.
  We claim that $g$ is an arithmetical leftmost path of $T$.
  Let $f \leq_{\T}^{\a} g$ and  $f \in [T]$. Since $g \in \M$ and $\M \models \ACAo$, $f \in \M$.
  Since $g$ is the leftmost$^\M$ path, $g \leq_l f$.
\end{proof}

From now on, for theories $T$ and $T'$, we write $T < T'$ to mean $T'$ proves $T + \Con(T)$ over $\ACAo$.
It is shown in \cite{Townser_TLPP} that
$\ALPP < \TLPP$.
More precisely, a general form of relative leftmost path principle $\Sigma^0_{\alpha}\LPP$ for a well order $\alpha$ is introduced, and it is proved that
$\Sigma^0_{\alpha+5}\LPP$ proves the consistency of $\Sigma^0_{\alpha+1}\LPP$. Since $\ALPP$ is provable from $\Sigma^0_{\omega + 1}\LPP$, the consistency of $\ALPP$ is provable from $\TLPP$.
Thus we have
\begin{theorem}
  $\beta^1_0\RFN(1) < \TLPP$.
\end{theorem}

On the other hand, $\TLPP$ is strictly weaker than $\beta^1_0\RFN(2)$.
To prove this result, we introduce a variant of $\beta^1_0\RFN(n)$.
\begin{definition}
  Let $\sigma$ be a $\Pi^1_2$ formula.
  We define $\beta^1_0\RFN(n;\sigma)$ as the following assertion:
  \begin{align*}
    \forall X \exists \M_0,\ldots,\M_{n}: \text{coded $\omega$-models }
    (&X \in \M_0 \in_\beta \M_1 \cdots \in_\beta \M_{n} \, \land \\
    &\M_n \models \ACAo +\, \sigma).
  \end{align*}
\end{definition}
\begin{lemma}\label{ATRo proves beta_sat(TLPP)}
  Over $\ATRo$, any $\beta$-model satisfies $\TLPP$.
\end{lemma}
\begin{proof}
  Let $\M$ be a $\beta$-model. Let $T,\alpha \in \M$ such that
  \begin{align*}
    \M \models \ &T \text{ is an ill-founded tree } \land \\
    &\alpha \text{ is a well order}.
  \end{align*}
  We show that $\M$ has a $\Delta^0_{\alpha}$ leftmost path.

  Let $S = \{ \sigma \in T: \exists f \in \M (\sigma \prec f \land f \in [T])\} $.
  Since $\M$ is a $\beta$-model, $S = \{\sigma \in T: \exists f (\sigma \prec f \in [T])\}$.
  Now, $S$ is a nonempty pruned tree and hence its leftmost path $g$ is computable from $S$.
  Moreover,  $g$ is the leftmost path of $T$.

  Now we have
    $\exists g \forall f \in [T] (g \leq_l f)$
  and hence
  \begin{equation*}
    \exists g \forall f \leq_{\T} g^{(\alpha)} (f \in [T] \to g \leq_l f).
  \end{equation*}
  The above condition is $\Sigma^1_1$. Therefore,
  \begin{equation*}
    \M \models \exists g \forall f \leq_{\T} g^{(\alpha)} (f \in [T] \to g \leq_l f).
  \end{equation*}
  This completes the proof.
\end{proof}

\begin{theorem}
  $\beta^1_0\RFN(1;\ATRo)$ proves the $\omega$-model reflection of $\TLPP$ over $\ACAo$.
  Therefore, $\TLPP < \beta^1_0\RFN(1;\ATRo)$.
\end{theorem}
\begin{proof}
  Assume $\ACAo + \beta^1_0\RFN(1;\ATRo)$. Let $X$ be a set.
  Take coded $\omega$-models $\M_0,\M_1$ such that
  \begin{align*}
    X \in \M_0 \in_{\beta} \M_1 \land \M_1 \models \ATRo.
  \end{align*}
  Then, since $\M_0$ is a coded $\beta$-submodel of $\M_1$, $\M_0 \models \TLPP$.
  This completes the proof.
\end{proof}

\begin{theorem}
  $\beta^1_0\RFN(2)$ proves the $\omega$-model reflection of $\beta^1_0\RFN(1;\ATRo)$.
  Therefore, $\beta^1_0\RFN(1;\ATRo) < \beta^1_0\RFN(2)$.
\end{theorem}
\begin{proof}
  Let $X$ be a set.
  By $\beta^1_0\RFN(2)$, take coded $\omega$-models $\M_0,\M_1,\M_2$ such that
  \begin{equation*}
    X \in \M_0 \in_\beta \M_1 \in_\beta \M_2 \land \M_2 \models \ACAo.
  \end{equation*}
  Then $\M_2 \models [\M_1 \models \ATRo]$ by Lemma \ref{beta-model satisfies atro}.
  Thus, for any $Y \in \M_0$,
  \begin{equation*}
    \M_2 \models \exists \mathcal{N}_0,\mathcal{N}_1 [Y \in \mathcal{N}_0 \in_\beta \mathcal{N}_1 \land \mathcal{N}_1 \models \ATRo].
  \end{equation*}
  Since $\M_0$ is a $\beta$-submodel of $\M_2$, we have
  \begin{equation*}
    \M_0 \models \forall Y \exists \mathcal{N}_0,\mathcal{N}_1 [Y \in \mathcal{N}_0 \in_\beta \mathcal{N}_1 \land \mathcal{N}_1 \models \ATRo].
  \end{equation*}
  Hence, $X \in \M_0$ and $\M_0 \models \beta^1_0\RFN(1;\ATRo)$.
  This completes the proof.
\end{proof}

Next, let us consider iterated versions of $\ALPP$.
\begin{definition}
  Let $n \in \omega$.
  We define $\Omega(n+1)$ as the following assertion: for any $X$ there are $f_0,\ldots,f_n$ such that
  \begin{align*}
    &f_0 \in [\Phi_T(X)] \land \bigwedge_{i < n} f_{i + 1} \in [\Phi_T(f_i)]  \, \land \\
    &\forall g \leq_{\T}^{\a} f_n [(g \in [\Phi_T(X)] \to f_0 \leq_l g) \land \bigwedge_{i < n} (g \in [\Phi_T(f_i)] \to f_{i+1} \leq_l g)].
  \end{align*}

  We also define $\It^{n+1} (\ALPP)$ as follows.
  Define $\Theta(T,f,g)$ be the following formula:
  \begin{align*}
    \Theta(T,f,g) \equiv \text{$T$ is a tree and $g \in [T]$ and } T,g \leq_{\T}^{\a} f.
  \end{align*}
  For each $m < n$, define $\psi^n_m$ by
\begin{align*}
  &\psi^n_0(T_0,\ldots,T_n,f_0,\ldots,f_n) \equiv \forall g \leq^{\a}_{\T} f_0 \oplus \cdots \oplus  f_{n} \bigwedge_{i \leq n} (g \in [T_i] \to f_i \leq_l g), \\
  &\psi^n_{m+1}(T_0,\ldots,T_{n-m-1},f_0,\ldots,f_{n-m-1}) \equiv \\
  &\forall T_{n-m},g_{n-m}(
  \Theta(T_{n-m},g_{n-m},f_{n-m-1})  \to \exists f_{n-m} \in [T_{n-m}] \psi^n_m(T_0,\ldots,T_{n-m},f_0,\ldots,f_{n-m})
  )
\end{align*}
Then,   define $\It^{n+1}(\ALPP)$ as follows.
\begin{align*}
  \It^{n+1}(\ALPP) \equiv \forall T_0 \text{ : ill-founded tree } \exists f_0 \in [T_0] \psi^n_{n}(T_0,f_0).
\end{align*}
\end{definition}
For example, $\It^{1}(\ALPP)$ is equivalent to $\ALPP$ and $\It^{2}(\ALPP)$
states that
for any ill-founded tree $T_0$, there exists a path $f_0 \in [T_0]$ such that
\begin{align*}
  \forall T_1,g_1(\Theta(T_1,g_1,f_0) \to \exists f_1 \in [T_1]\forall g \leq_{\T}^{\a}f_0 \oplus f_1\bigwedge_{i \leq 1}( g \in [T_i] \to f_{i} \leq_l g)).
\end{align*}

For simplicity, we write $(T,g) \in \Theta(f)$ instead of $\Theta(T,f,g)$. Then,
$\It^{n+1}(\ALPP)$ states that
for any ill-founded tree $T_0$, there exists a path $f_0 \in [T_0]$ such that
\begin{align*}
  \forall (T_1,g_1) \in \Theta(f_0) \exists f_1 \in [T_1] \cdots
  \forall (T_n,g_n) \in \Theta(f_{n-1}) \exists f_n \in [T_n]\\
  \forall g \leq_{\T}^{\a} f_0 \oplus \cdots \oplus f_n \bigwedge_{i \leq n } (g \in [T_i] \to f_i \leq_l g).
\end{align*}

\begin{theorem}\label{ItALPP and hyperjumps}
  For $n \in \omega$, the following assertions are equivalent over $\ACA_0^+$.
  \begin{enumerate}
    \item $\beta^1_0\RFN(n+1)$,
    \item $\Omega(n+1)$,
    \item $\It^{n+1}(\ALPP)$.
   \end{enumerate}
\end{theorem}
\begin{proof}
  ($1 \to 3$)
  Let $T_0$ be an ill-founded tree and $h$ be a path of $T_0$.
  By $\beta^1_0\RFN(n+1)$ take a sequence of $\omega$-models such that
  \begin{equation*}
    T_0,h \in \M_0 \in_\beta \cdots \in_\beta \M_{n+1} \land \M_{n+1} \models \ACAo.
  \end{equation*}
  Now $\M_1$ satisfies ``$T_0$ has a leftmost path".
  Take the leftmost$^{\M_1}$ path $f_0$.
  Then, $\M_{n+1}$ also satisfies ``$f_0$ is the leftmost path of $T_0$".
  Let $T_1,g_1 \leq_{\T}^{\a} f_0$ such that $T_1$ is a tree and $g_1 \in [T_1]$.
  Then, $\M_2$ satisfies ``$T_1$ has a leftmost path".
  Let $f_1$ be the leftmost$^{\M_2}$ path.
  Then, $\M_{n+1}$ satisfies ``$f_1$ is the leftmost path of $T_1$".
  Continuing this argument, we have that
  \begin{align*}
    &\forall (T_1,g_1) \in \Theta(f_0)\exists f_1 \in [T_1] \cdots \forall (T_n,g_n) \in \Theta(f_{n-1})\exists f_n \in [T_n]\\
    &\bigwedge_{i \leq n} (\M_{n+1} \models ``f_i \text{ is the leftmost path of } T_i").
  \end{align*}
  Since $\M_{n+1}$ is an $\omega$-model of $\ACAo$, it is closed under arithmetical reduction. This completes the proof.

  ($3 \to 2$) This is trivial.

  ($2 \to 1$) Let $X$ be a set. We show that there is an $\omega$-model $\M$ such that
  $\M \models \ACAo + \exists Y (Y = \HJ^{n+1}(X))$.
  Let $f_0,\ldots,f_{n}$ be witnesses of $\Omega(n)$ for $X$.
  Then, $f_0 \leq_{\T} \cdots \leq_{\T} f_n$.
  By $\ACA_0^+$, take the smallest $\omega$-model $\M$ of $\ACAo$ containing $f_{n}$.
  Then, $\M$ satisfies ``$f_0$ is the leftmost path of $\Phi_T(X)$" and hence $f_0$ computes $\HJ(X)$ in $\M$.
  Similarly, $f_i$ is the leftmost path of $\Phi_T(f_{i-1})$ in $\M$ and hence it computes $\HJ(f_i)$ in $\M$.
  Thus, $\M \models \exists Y (Y = \HJ^{n+1}(X))$.
\end{proof}

As shown in \cite{suzuki_yokoyama_fp}, $\Delta^0_n\LPP$s and variants of $\beta$-model reflection called $\Delta^0_n\beta$-model refections
are deeply related.
In the following, we see the relationship of variants of leftmost path principles and variants of $\beta$-models reflection.
\begin{remark}
  We can show that
  \begin{itemize}
    \item for any $n \in \omega$, $\ACAo$ proves that any $\beta$-model satisfies $\Delta^0_n\LPP$,
    \item $\ACA_0'$ proves that for any $n \in \N$, any $\beta$-model satisfies $\Delta^0_n\LPP$,
    \item $\ACA_0^+$ proves that any $\beta$-model satisfies $\ALPP$.
  \end{itemize}
  by the same proof of Lemma \ref{ATRo proves beta_sat(TLPP)}.
\end{remark}
In contrast to this result, even $\ACA_0^+$ does not prove that
any $\beta$-model satisfies $\TLPP$.
To prove this, we can use the same proof as in the following lemma.
\begin{lemma}\label{characterization of TLPP by reflection}
  $\TLPP$ is equivalent to
  \begin{align*}
    \forall \alpha : \WO \forall X \exists \M
    (&\alpha,X \in \M, \\
    &\M \models \ACAo + \exists Y(Y= \HJ(X)) ), \\
    &\text{$\M$ is closed under $\alpha$-jump}).
  \end{align*}
\end{lemma}
\begin{proof}
  First, assume $\TLPP$.
  Take a well order $\alpha$. Then, $\alpha \cdot \omega$ is also a well-ordering
  \footnote{Here, $\alpha \cdot \omega$ denotes the order type of lexicographical ordering of $\N \times \alpha$.}.
  We note that $\alpha \cdot \omega$ is closed under $+\alpha$ in the following sense:
  Recall that  $|L|$ denotes the field $\{i : (i,i) \in L\}$ of $L$ for a linear order $L$. For each $(n, i) \in |\alpha \cdot \omega|$,
  there is a certain embedding from $\alpha$ to the interval $[(n,i),(n+2,0)]$ defined by $f(j) = (n+1,j)$.
  Since $\ATR_0$ holds, there is an embedding $g$ from $\alpha$ to an initial segment of $[(n,i),(n+2,0)]$.
  Then, $\sup g$ is regarded as $(n,i) + \alpha$.

  Take an arbitrary $X$. We show that
  \begin{align*}
  \exists \M
  (&\alpha,X \in \M, \\
  &\M \models \ACAo + \exists Y(Y= \HJ(X)) ), \\
  &\text{$\M$ is closed under $\alpha$-jump}).
  \end{align*}
  Recall that there is a total Turing functional $\Phi_T$ which outputs an ill-founded tree for any input
  (see Lemma \ref{lem-of-Phi_T}).
  By $\TLPP$, take a $\Delta^0_{\alpha\cdot\omega}$-leftmost path $f$ of $\Phi_T(\alpha \oplus X)$.
  Define
  \begin{equation*}
    \M = \{A : A \leq_{\T} f^{(n,i)} \text{ for some } (n,i) \in |\alpha \cdot \omega|\}.
  \end{equation*}
  Since $\M \models[\text{$f$ is a leftmost path of $\Phi_T(\alpha \oplus X)$}]$, $\M \models \exists Y(Y = \HJ(X))$.
  Moreover, by construction, $\M \models \ACAo$ and closed under $\alpha$-jump.

  The converse is trivial because the leftmost path in $\M$ is an $\Delta^0_{\alpha}$-leftmost path in the ground model.
\end{proof}
Recall that a coded $\omega$-model of $\ACAo$ satisfies $\ACA_0^+$ if it is closed under taking the $\omega$-jump of a set.
Therefore, by the same proof as above, we have
\begin{lemma}
  Over $\ACAo$, $\Delta^0_{\omega^2}\LPP$ proves $\beta^1_0\RFN(1;\ACA_0^+)$.
\end{lemma}
\begin{theorem}
  Over $\ACAo$, $\TLPP$ is strictly stronger than $\beta^1_0\RFN(1;\ACA_0^+)$.
\end{theorem}
\begin{proof}
  It is immediate from the fact that $\TLPP$ is strictly stronger than $\Delta^0_{\omega^2}\LPP$.
\end{proof}
\begin{corollary}
  $\ACA_0^+$ does not prove that any $\beta$-model satisfies $\TLPP$.
\end{corollary}
\begin{proof}
  For the sake of contradiction, assume $\ACA_0^+$ proves that any $\beta$-model satisfies $\TLPP$.

  We work in $\ACAo + \beta^1_0\RFN(1;\ACA_0^+)$ and show that $\Con(\TLPP)$ holds.
  Take coded $\omega$-models $\M_0,\M_1$ such that
  $\M_0 \in_{\beta} \M_1$ and $\M_1 \models \ACA_0^+$.
  Then, $\M_1 \models [\M_0 \models \TLPP]$ by assumption and hence $\M_1 \models \Con(\TLPP)$.
  Since $\Con(\TLPP)$ is arithmetical, $\Con(\TLPP)$ holds.
  However, this is impossible because $\TLPP$ implies $\beta^1_0\RFN(1;\ACA_0^+)$ over $\ACAo$.
\end{proof}

\begin{theorem}
  Over $\ACAo$,
$\beta^1_0\RFN(n;\ALPP)$ implies the $\omega$-model reflection of
$\beta^1_0\RFN(n+1)$.
\end{theorem}
\begin{proof}
  We reason in $\ACAo + \beta^1_0\RFN(n;\ALPP)$.
  Take an arbitrary set $A$.
  By $\beta^1_0\RFN(n;\ALPP)$, take coded $\omega$-models $\M_0,\ldots,\M_n$ such that
  \begin{align*}
    A \in \M_0 \in_\beta \cdots \in_\beta \M_{n-1} \in_\beta \M_n \land \M_n \models \ACAo + \ALPP.
  \end{align*}
  Since $\ALPP$ and $\beta^1_0\RFN(1)$ is equivalent over $\ACAo$, $\M_n \models \beta^1_0\RFN(1)$.

  We show that $\M_0 \models \beta^1_0\RFN(n+1)$.
  By $\beta^1_0\RFN(1)$ in $\M_n$, there are coded $\omega$-model $\NN_0,\NN_1 \in \M_n$ such that
  \begin{align*}
    \M_n \models [\M_{n-1} \in \NN_0 \in_\beta \NN_1 \land \NN_1 \models \ACAo].
  \end{align*}
  Since $\M_{n-1}$ is a $\beta$-submodel of $\M_n$ and $\NN_0$ is an $\omega$-submodel of $\M_n$, for any $\Pi^1_1$ sentence $\sigma$ with parameters from $\M_{n-1}$,
  \begin{align*}
    [\M_{n-1} \models \sigma] &\to [\M_n \models \sigma] \\
    &\to [\NN_0 \models \sigma].
  \end{align*}
  Thus, $\M_{n-1}$ is a $\beta$-submodel of $\NN_0$.
  Therefore, for any $X \in \M_0$,
  \begin{align*}
    \M_n \models [X \in \M_0 \in_{\beta} \cdots \in_{\beta} \M_{n-1} \in_\beta \NN_0 \in_\beta \NN_1 \land \NN_1 \models \ACAo].
  \end{align*}
  Therefore, for any $X \in \M_0$,
  \begin{align*}
    \M_n \models \exists \HH_0,\ldots,\HH_{n+1}
    [X \in \HH_0 \in_{\beta} \cdots \in_{\beta} \HH_{n-1} \in_\beta \HH_n \in_\beta \HH_{n+1} \land \HH_{n+1} \models \ACAo].
  \end{align*}
  Since $\M_0$ is a $\beta$-submodel of $\M_n$, for any $X \in \M_0$,
  \begin{align*}
    \M_0 \models \exists \HH_0,\ldots,\HH_{n+1}
    [X \in \HH_0 \in_{\beta} \cdots \in_{\beta} \HH_{n-1} \in_\beta \HH_n \in_\beta \HH_{n+1} \land \HH_{n+1} \models \ACAo].
  \end{align*}
  Thus, $\M_0 \models \beta^1_0\RFN(n+1)$.
\end{proof}
%

The point of the above proof is that $\ALPP$ implies $\beta^1_0\RFN(1)$ and hence $\M_n$ satisfies $\beta^1_0\RFN(1)$.
Since $\TLPP$ implies $\beta^1_0\RFN(1;\ACA_0^+)$,
we can also show the following theorem by the same way as the above proof.
\begin{theorem}
  Over $\ACAo$,
$\beta^1_0\RFN(n;\TLPP)$ implies the $\omega$-model reflection of
$\beta^1_0\RFN(n+1;\ACA_0^+)$.
\end{theorem}

\begin{theorem}
  Over $\ACAo$, $\beta^1_0\RFN(n+1)$ proves the $\omega$-model reflection of
  $\beta^1_0\RFN(n;\forall k. \Delta^0_k\LPP)$.
  Here, $\forall k. \Delta^0_k\LPP$ is the assertion that
  \begin{align*}
    \forall k \forall T : \text{ill-founded tree } \exists g \in [T]
    \forall h \leq_{\T} g^{(k)} (h \in [T] \to g \leq_l h).
  \end{align*}
\end{theorem}
\begin{proof}
  We reason in $\ACAo + \beta^1_0\RFN(n+1)$.
  Take an arbitrary set $A$.
  By $\beta^1_0\RFN(n+1)$, take coded $\omega$-models $\M_0,\ldots,\M_{n+1}$ such that
  \begin{align*}
    A \in \M_0 \in_\beta \cdots \in_\beta \M_n \in_\beta \M_{n+1} \land \M_{n+1} \models \ACAo.
  \end{align*}
  We show that $\M_0 \models \beta^1_0\RFN(n;\forall k.\Delta^0_k\LPP)$.
  Since $\M_{n+1} \models \ACA_0'$, $\M_{n+1} \models [\M_n \models \forall k.\Delta^0_k\LPP]$.
  Therefore, for any $X \in \M_0$,
  \begin{align*}
    \M_{n+1} \models [X \in \M_0 \in_\beta \cdots \in_\beta \M_n \land \M_n \models \forall k.\Delta^0_k\LPP].
  \end{align*}
  Hence, for any $X \in \M_0$,
  \begin{align*}
    \M_{n+1} \models \exists \HH_0 \cdots \HH_n [X \in \HH_0 \in_\beta \cdots \in_\beta \HH_n \land \HH_n \models \forall k.\Delta^0_k\LPP].
  \end{align*}
  Since $\M_0$ is a $\beta$-submodel of $\M_{n+1}$, for any $X \in \M_0$,
  \begin{align*}
    \M_{0} \models \exists \HH_0 \cdots \HH_n [X \in \HH_0 \in_\beta \cdots \in_\beta \HH_n \land \HH_n \models \forall k.\Delta^0_k\LPP].
  \end{align*}
  Therefore, $\M_0 \models \beta^1_0\RFN(n,\forall k. \Delta^0_k\LPP)$.
\end{proof}

Recall that any coded $\beta$-model satisfies $\ALPP$ over $\ACA_0^+$.
Thus, by the same proof as above, we have
\begin{theorem}
  Over $\ACAo$, $\beta^1_0\RFN(n+1;\ACA_0^+)$ proves the $\omega$-model reflection of $\beta^1_0\RFN(n;\ALPP)$.
\end{theorem}

Summarizing previous theorems, we have a hierarchy between $\beta^1_0\RFN(n+1)$ and $\beta^1_0\RFN(n+2)$ as follows.
\begin{corollary}\label{cor:summary}
  Over $\ACAo$, the following holds.
  \begin{align*}
     \beta^1_0\RFN(n+1) &< \beta^1_0\RFN(n;\ALPP) < \beta^1_0\RFN(n+1;\ACA_0^+) \\
     &< \beta^1_0\RFN(n;\TLPP) < \beta^1_0\RFN(n+1;\ATRo)  \\
     &< \beta^1_0\RFN(n+1;\forall k. \Delta^0_k\LPP) < \beta^1_0\RFN(n+2)
  \end{align*}
\end{corollary}

\begin{remark}
  In a forthcoming paper \cite{suzuki2024relative}, the first author introduced a characterization of $\ALPP,\TLPP,\Delta^0_k\LPP$ and $\beta^1_0\RFN(n)$ in terms of the $\omega$-model reflection of transfinite inductions as follows.
  Over $\RCAo$, we have
  \begin{itemize}
    \item $\ALPP \leftrightarrow \RFN(\Pi^1_{\infty}\hyp\TI)$\footnote{This equivalence is independently proved by Freund \cite{Freund-Fraisse}.},
    \item $\TLPP \leftrightarrow \beta^1_0\RFN(1;\Pi^1_1\hyp\TI)$,
    \item $\Delta^0_k\LPP \leftrightarrow \RFN(\Pi^1_{k+1}\hyp\TI)$ for $k > 1$,
    \item $\beta^1_0\RFN(n+1) \leftrightarrow \beta^1_0\RFN(n;\Pi^1_{\infty}\hyp\TI)$.
  \end{itemize}
  Here, $\RFN(T)$ denotes the $\omega$-model reflection of a theory $T$.
  These characterization may be helpful to understand the structure of the hierarchy in Corollary \ref{cor:summary}.
\end{remark}

\begin{remark}
  We note that the gap between each pair of the hierarchy in Corollary \ref{cor:summary} is relatively large.
  The proofs for theses inequalities actually show that the stronger one implies not only the omega-model reflection of the lower one, but also the omega-model reflection of [the lower one and transfinite induction for $\Pi^1_1$ formulas ($\Pi^1_1\myhyphen\TI$)].
  Since, for any $\Pi^1_2$ sentence $\sigma$,
  the omega-model reflection of [$\sigma + \Pi^1_1\myhyphen\TI]$ proves any time iteration of the omega-model reflection of $\sigma$ as in \cite[Section 5]{suzuki_yokoyama_fp}, we conclude that any time iterations of the omega-model reflection of the lower one is weaker than the stronger one.
\end{remark}

In the remaining of this section, we give another characterization of $\ALPP$.
\begin{definition}
  Let $\M$ be a coded $\omega$-model of $\ACAo$.
  We say $\M$ is an $\A\beta$-models if for any $\Pi^0_2$ formula $\theta(X)$,
  \begin{equation*}
    \exists X \leq_{\T}^{\a} \M \theta(X) \leftrightarrow \M \models \exists X \theta(X).
  \end{equation*}
\end{definition}
\begin{lemma}
  Let $\varphi(X)$ be an arithmetical formula with exactly displayed free variable.
  Then, $\ACAo$ proves that for any $\A\beta$-model $\M$,
  \begin{equation*}
    \exists X \leq_{\T}^{\a} \M \theta(X) \leftrightarrow \M \models \exists X \theta(X).
  \end{equation*}
\end{lemma}
\begin{proof}
  Let $\varphi(X)$ be an arithmetical formula. Then there exists a number $k$ and a $\Pi^0_2$ formula $\theta(X,f)$ such that
  $\ACAo$ proves
  \begin{align*}
    \forall X (\varphi(X) \leftrightarrow \exists f\leq_{\T} X^{(k)} \theta(X,f)).
  \end{align*}
  We work in $\ACAo$. Assume $\M$ is an $\A\beta$-model and $\exists X \leq_{\T}^{\a} \M \varphi(X)$.
  Since $\phi$ is arithmetical, $\phi(X)$ actually holds. Thus, there exists $f \leq_{\T} X^{(k)}$ such that
  $\theta(X,f)$. Since $\M$ is a model of $\ACAo$, $\M \models \exists f \leq_{\T} X^{(k)} \theta(X,f)$.
  Hence $\M \models \varphi(X)$.
\end{proof}

\begin{definition}
  Define the $\A\beta$-model reflection as the assertion that
  for any set $X$ there is an $\A\beta$-model containing $X$.
\end{definition}

\begin{theorem}\label{beta-rfn and Abeta}
  Over $\ACAo$, $\beta^1_0\RFN(1)$ is equivalent to $\A\beta$-model reflection.
\end{theorem}
\begin{proof}
  First, assume $\beta^1_0\RFN(1)$.
  Let $X$ be a set.
  By $\beta^1_0\RFN(1)$, take a coded $\omega$-models $\M_0,\M_1$ such that
  $X \in \M_0 \in_\beta \M_1 \land \M_1 \models \ACAo$.

  We claim that $\M_0$ is an $\A\beta$-model.
  Let $\theta(Y)$ be an arithmetical formula.
  Assume $\exists Z \leq_{\T}^{\a} \M_0 \theta(Z)$.
  Then, there exists $n \in \N$ such that $\exists Z \leq_{\T} (\M_0)^{(n)} \theta(Z)$.
  Since $\M_0 \in \M_1$ and $\M_1 \models \ACA_0'$, $\M_1 \models \exists Z \theta(Z)$.
  Since $\M_0$ is a $\beta$-submodel of $\M_1$, $\M_0 \models \exists Z \theta(Z)$.

  We next assume $\A\beta$-model reflection.
  Then, $\ACA_0^+$ holds.
  Let $X$ be a set.
  By $\A\beta$-model reflection, take an $\A\beta$-model $\M_0$ such that
  $X \in \M_0$. By $\ACA_0^+$, take $\M_1 = \{Y : Y \leq_{\T}^{\a} \M_0$\}.

  Now $X \in \M_0 \in \M_1 \land \M_1 \models \ACAo$. Thus, it is enough to show that
  $\M_0$ is a $\beta$-submodel of $\M_1$.
  Let $\theta(Y)$ be an arithmetical formula.
  Assume $\M_1 \models \exists Y \theta(Y)$.
  Then, $\exists Y \leq_{\T}^{\a} \M_0 \theta(Y)$ and hence $\M_0 \models \exists Y \theta(Y)$.
  This completes the proof.
\end{proof}

\begin{corollary}
  ($\ACAo$.)
  $\ALPP$ is equivalent to $\A\beta$-model reflection.
\end{corollary}
\begin{proof}
  Since $\ALPP$ is equivalent to $\beta^1_0\RFN(1)$ and $\beta^1_0\RFN(1)$ is equivalent to $\A\beta$-model reflection,
  $\ALPP$ is equivalent to $\A\beta$-model reflection.
\end{proof}

For each $n > 0$, define $\A\beta\RFN(n)$ as the assertion that
  \begin{align*}
    \forall X \exists \M_0 \cdots \M_{n-1}(X \in_{\beta} \cdots \in_{\beta} \M_{n-1} \land \M_{n-1} \text{ is an $\A\beta$-model}).
  \end{align*}
Then, by the same argument as in Theorem \ref{beta-rfn and Abeta}, we have
\begin{theorem}
  Over $\ACAo$, $\beta^1_0\RFN(n)$ is equivalent to $\A\beta\RFN(n-1)$ for any $n >0$.
\end{theorem}

\section{The arithmetical relativization}
As we have noted,
the $\beta^1_0\RFN$-hierarchy reveals that
how many times the hyperjump operator is used to prove a $\Pi^1_2$ sentence.
In addition, we can classify $\Pi^1_3$ sentences provable from $\Pi^1_1\myhyphen\CAo$ by comparing
their $\Pi^1_2$ approximations and the $\beta^1_0\RFN$-hierarchy.
In this section, we summarize the idea of $\Pi^1_2$ approximations of $\Pi^1_3$ sentences.

\begin{definition}
  Let $\sigma \equiv \forall X (\theta(X) \to \exists Y \forall Z \eta(X,Y,Z)$ be a $\Pi^1_3$ sentence.
  Define $\rel(\sigma)$ by $\forall X(\theta(X) \to \exists Y \forall Z \leq_{\T}^{\a} (X \oplus Y) \eta(X,Y,Z)$.
  We call $\rel(\sigma)$ the arithmetical relativization of $\sigma$.
\end{definition}
For example, $\ALPP$ and $\A\beta\RFN(n)$ are instances of $\rel(\sigma)$.
Indeed, if we take $\theta(X)$ to be `$X$ is an ill-founded tree'
and $\eta(X,Y,Z)$ to be $(Y \in [X]) \land (Z \in [X] \to Y \leq_l X)$, then
$\ALPP$ is equivalent to $\rel(\forall X(\theta(X) \to \exists Y \forall Z \eta(X,Y,Z)))$.

\begin{remark}
  In the case of $\ALPP$, the following $\rel'(\sigma)$ is equivalent to $\rel(\sigma)$.
  \begin{equation*}
    \rel'(\sigma) \equiv \forall X (\theta(X) \to  \exists Y \forall Z \leq_{\T}^{\a} Y \eta(X,Y,Z)).
  \end{equation*}
  In general, $\rel(\sigma)$ and $\rel'(\sigma)$ are not equivalent. The equivalence of $\rel(\sigma)$ and $\rel'(\sigma)$ is related to the notion of \textit{cylinder} in Weihrauch reduction.
\end{remark}

We note that if $\sigma$ is a $\Pi^1_3$ sentence provable from $\Pi^1_1\myhyphen\CAo$,
then $\rel(\sigma)$ is a $\Pi^1_2$ sentence provable from $\Pi^1_1\myhyphen\CAo$.
Thus, $\rel(\sigma)$ is provable from $\ACAo + \beta^1_0\RFN(n)$ for some $n$.
Let us write $n_{\sigma}$ for the smallest $n$ such that $\ACAo + \beta^1_0\RFN(n) \vdash \rel(\sigma)$ for such a $\sigma$.
Then, we can classify $\Pi^1_3$ sentences provable from $\Pi^1_1\myhyphen\CAo$ according to $n_{\sigma}$.

We then give a lemma which is useful to evaluate $n_{\sigma}$.
\begin{definition}
  Let $\varphi \equiv \forall X \exists Y \theta$ and $\psi \equiv \forall X \exists Y \eta$.
  For a theory $T$, we say $T$ proves \textit{the lightface implication of $\varphi$ to $\psi$} if
  \begin{equation*}
    T \vdash \forall X(\exists Y \theta \to \exists Z \eta).
  \end{equation*}
\end{definition}

\begin{lemma}
  Let $\varphi \equiv \forall X \exists Y \forall Z \theta(X,Y,Z)$ and $\psi \equiv \forall X \exists Y \forall Z \eta(X,Y,Z)$ be $\Pi^1_3$ sentences such that
  $\ACAo$ proves the lightface implication of $\varphi$ to $\psi$.
  Then, $\ACA_0^+$ proves $\rel(\varphi) \to \rel(\psi)$.
\end{lemma}
\begin{proof}
  We reason in $\ACA_0^+ + \forall X \exists Y \forall Z \leq_{\T}^{\a} X \oplus Y \theta(X,Y,Z)$.
  Let $X,Y$ be such that $\forall Z \leq_{\T}^{\a} X \oplus Y \theta(X,Y,Z)$.
  Take the smallest $\omega$-model $\M$ of $\ACA_0$ containing $X \oplus Y$.
  Then, $\M \models \exists U \forall Z \theta(X,U,Z)$ via  $U = Y$.
  Since $\M$ is a model of $\ACAo$, $\M \models \exists V \forall Z \eta(X,V,Z)$. Take such a $V$.
  Thus, $\exists W \forall Z \leq_{\T}^{\a} X \oplus W \eta(X,W,Z)$ via  $W = V$.
\end{proof}
For example, if $\ACAo$ proves $\forall X(\exists Y \forall Z \theta(X,Y,Z) \to \exists W(W = \HJ(X)))$,
then $\ACA_0^+$ proves $\rel(\forall X \exists Y \forall Z\theta) \to \beta^1_0\RFN(1)$.
Conversely, if $\ACAo$ proves
$\forall X(\exists W(W = \HJ(X)) \to \exists Y \forall Z \theta(X,Y,Z))$, then
$\ACA_0^+$ (and hence even $\RCAo$) proves $\beta^1_0\RFN(1) \to \rel(\forall X \exists Y \forall Z\theta)$.

\section{Pseudo Ramsey's theorem and $\beta\RFN$}
In this section, we introduce $\Pi^1_2$ variants of Ramsey's theorem for $[\N]^{\N}$.
Our variants may be seen as an instance of the arithmetical relativization of usual Ramsey's theorem.
In this section, we see the relationship
between them and  the $\beta^1_0\RFN$-hierarchy.

Let $[X]^{\N}$ denotes the class of all infinite subsets of $X$.
Then, Ramsey's theorem for $[\N]^{\N}$, also known as Galvin-Prikry's theorem, claims that certain subclasses of $[\N]^{\N}$ has Ramsey's property in the sense that
a class $\mathcal{A} \subseteq [\N]^{\N}$ has Ramsey's property if there is an infinite set $H \subseteq \N$ such that
$[H]^{\N}$ is a subset of $\mathcal{A}$ or its complement.

We first recall how to formalize Ramsey's theorem in second-order arithmetic and some existing results.
For the details, see also \cite{Simpson}.
\begin{definition}
  Let $[\N]^{\N}$ denote the class of all strictly increasing sequence of natural numbers.
  Let $\varphi(f,X)$ be a formula for $f \in [\N]^{\N}$ and $X \subseteq \N$. We say $\varphi$ has Ramsey's property if
  \begin{equation*}
    \forall X \exists h \in [\N]^{\N}\,\RP_{\varphi}(h,X),
  \end{equation*}
  where $\RP_{\varphi}(h,X)$ is
  \begin{equation*}
    \forall g \in [\N]^{\N} \varphi(h \circ g,X) \lor \forall g \in [\N]^{\N} \lnot \varphi(h \circ g,X).
  \end{equation*}
\end{definition}
In this definition, we identify $f \in [\N]^{\N}$ and its range, $\varphi(f,X)$ and the class $\{f \in [\N]^{\N} : \varphi(f,X)\}$.
In particular, $h \in [\N]^{\N}$ corresponds to an infinite subset $H$ of $\N$, $h \circ g$ corresponds to an infinite subset of $H$.

\begin{definition}
  Let $\Gamma$ be a class of formulas.
  Then, $\Gamma$-Ramsey's theorem (write $\Gamma\myhyphen\Ram)$ is the assertion that any $\varphi \in \Gamma$ has Ramsey's property.
\end{definition}

\begin{theorem}
  \cite[VI.6.4]{Simpson}
  Over $\RCAo$, $\Pi^1_1\myhyphen\CAo$ is equivalent to $\Sigma^1_0\myhyphen\Ram$.
\end{theorem}

We give a refined analysis for this theorem by using the $\beta^1_0\RFN$-hierarchy.
We start with giving $\Pi^1_2$ variants of Ramsey's property which we call pseudo-Ramsey's property.

\begin{definition}\label{definition-of-rel-ram}
  Let $\varphi(f,X)$ be a formula for $f \in [\N]^{\N}$ with exactly displayed free variables.
  We say $\varphi(f,X)$ has pseudo-Ramsey's property if
  \begin{align*}
    \forall X \exists h \in [\N]^{\N}(\pRP_{\varphi}(h,X))
  \end{align*}
  where $\pRP$ is
  \begin{align*}
    & \forall g \in [\N]^{\N} (g \leq^{\a}_{\T} f \oplus X \to  \varphi(f \circ g)) \\
   \lor \,\,  &\forall g \in [\N]^{\N} (g \leq_{\T}^{\a} f \oplus X \to  \lnot \varphi(f \circ g)).
  \end{align*}
\end{definition}

  The condition $\varphi$ has pseudo-Ramsey's property is the same as $\rel(\forall X \exists h \in [\N]^{\N}\RP_{\varphi})$ holds.
  Therefore, we write $\rel(\Gamma\myhyphen\Ram)$ for the set of sentences
  $\{\forall X \exists h \in [\N]^{\N}(\pRP_{\varphi}(h,X)) : \varphi \in \Gamma\}$.

We note that
for each $n >0$, there exists a formula $\varphi \in \Sigma^0_n$ such that
$\Sigma^0_n\myhyphen\Ram$ is axiomatizable by $\forall X \exists h \in [\N]^{\N}\, \RP_{\varphi}(h,X)$ over $\ACAo$, and hence
$\rel(\Sigma^0_n\myhyphen\Ram)$ is axiomatizable by  $\rel(\forall X \exists h \in [\N]^{\N}\,  \RP_{\varphi}(h,X))$ over $\ACAo$.
Since $\Sigma^0_n\myhyphen\Ram$ is provable from $\Pi^1_1\myhyphen\CAo$,  $\rel(\Sigma^0_n\myhyphen\Ram)$
is provable from $\beta^1_0\RFN(p(n))$ for some $p(n)$.
We show that $p(n) =n$. It follows from the following lemma.
\begin{lemma}
  (\cite{Simpson}, VI.6.2.)
  ($\ACAo$)
  Let $\M_0 ,\ldots, \M_{k-1}$ be coded $\beta$ models such that
  $\M_0 \in \cdots \in \M_{k-1}$.
  Then, for any $\Sigma^0_k$ formula $\varphi(f)$ with parameters from $\M_0$, $\varphi(f)$ has Ramsey's property via some $h \in \M_{k-1}$.
\end{lemma}

\begin{theorem}\label{betarfn to ram}
  $\rel(\Sigma^0_k\Ram)$ is provable from $\beta^1_0\RFN(k)$.
\end{theorem}
\begin{proof}
  Assume $\beta^1_0\RFN(k)$.
  Let $\varphi(f,X)$ be a $\Sigma^0_k$ formula for $f \in [\N]^{\N}$ with exactly displayed set variables.
  By $\beta^1_0\RFN(k)$, take coded $\omega$-models $\M_0,\ldots,\M_k$ such that
  $X \in \M_0 \in_\beta \cdots \in_\beta \M_{k-1} \in_\beta \M_k$ and $\M_k \models \ACAo$.
  Now, by the previous lemma, there exists $f \in \M_{k-1}$ such that
  \begin{equation*}
    \M_k \models \forall g \varphi(f \circ g) \lor \forall g \lnot \varphi(f \circ g).
  \end{equation*}
  Since $\M_k$ is a model of $\ACAo$, $\varphi$ has pseudo-Ramsey's property via $f$.
\end{proof}

In the remaining of this section, we show that $\ACA_0^+ + \{\rel(\Sigma^0_n\Ram) : n \in \omega\}$ proves
$\beta^1_0\RFN(n)$ for all $n$.
In \cite{Tanaka-Ramsey}, it is shown that
over $\ACAo$, $(\text{lightface-}\Delta^{0,X}_2)\Ram \to \exists Y(Y = \HJ(X))$.
Here $\text{lightface-}\Delta^{0,X}_2$ is the set of $\Delta^0_2$ formulas having no set parameters other than $X$.
Moreover, it is commented  without proofs that this result can be extended to the transfinite level: $\ATRo \vdash (\text{lightface-}\Delta^{1}_1\Ram) \to \forall \alpha \text{: recursive well order }\exists Y (Y = \HJ^{\alpha}(\varnothing)))$.
In this section, we give an explicit proof for finite level.
\begin{definition}
  Let $T$ be a tree.
  We say $f \in [\N]^{\N}$ majorizes $T$ if $\exists g \in [T] \forall n (g(n) \leq f(n))$.
\end{definition}

\begin{definition}
  Let $f$ be a function and $n \in \N$.
  Define $f_{+n}$ by $f_{+ n}(x) = f(x + n)$.
\end{definition}

\begin{lemma}
  ($\ACAo$)
  Let $T$ be a tree and  $f \in [\N]^{\N}$.
  Then, $f$ majorizes $T$ if and only if $\forall n > 0 \exists \tau \in T (|\tau| = n \land \forall m < n (\tau(m) \leq f(m)))$.
  In particular, $f$ majorizes $T$ is an arithmetical condition.
\end{lemma}
\begin{proof}
  It is easy by K\"{o}nig's lemma. For the details,
  see the proof of Lemma V.9.6. in \cite{Simpson}.
\end{proof}

\begin{definition}
  ($\ACAo$)
  Let $\theta(n,A)$ be the $\Sigma^1_1$ formula defining $\HJ(A)$.
  Write $\theta(n,A) \equiv \exists f \in \N^{\N} \forall y R(n,f[y],A[y])$  by a $\Sigma^0_0$ formula $R$.

  For each $n \in \N$ and $A \subseteq \N$, define a tree $T(n,A) = \{\sigma : \forall y < |\sigma| R(n,\sigma[y],A[y])\}$.
  For each $f \in [\N]^{\N}$ and $A \subseteq \N$, define $F(f,A) = \{n : \exists p (f_{+ p} \text{ majorizes } T(n,A))$.
\end{definition}

\begin{definition}\label{def of phi-star}
    Define $\Phi_{\ast}(f,A)$ by
    \begin{equation*}
      \forall n < f(0) (f_{+2} \text{ majorizes } T(n,A) \to f_{+1} \text{ majorizes } T(n,A) ).
    \end{equation*}
\end{definition}

\begin{lemma}\label{lightface ram to hj}
  (Essentially \cite{Simpson} VI.6.1)
  $\ACAo$ proves that for any $A \subseteq \N$ and $h  \in [\N]^{\N}$,
  if $\forall g(\Phi_{\ast}(h \circ g,A))$, then $F(h,A) = \HJ(A)$.
\end{lemma}
\begin{proof}
  We first remark that for each $n$ and $A$, $T(n,A)$ has a path if and only if $n \in \HJ(A)$.

  Take $h \in [\N]^{\N}$ such that $\forall g(\Phi_{\ast}(h \circ g,A))$.
  We show that $F(h,A) \subseteq \HJ(A)$.
  If $n \in F(h,A)$ then $f_{+p}$ majorizes $T(n,A)$ for some $p$, and hence $[T(n,A)] \neq \varnothing$. Thus, $n \in \HJ(A)$.

  Conversely, assume $n \in \HJ(A)$ and show that $n \in F(h,A)$. Now, $T(n,A)$ has a path. We claim that $h_{+(n+2)}$ majorizes $T(n,A)$.
  Let $g_0 \in [T(n,A)]$ and define $g(n) = g_0(0) + \cdots + g_0(n)$. Then $g \in [\N]^{\N}$ and it majorizes $T(n,A)$.
  Assume $h_{+(n+2)}$ does not majorize $T(n,A)$. Take $k>0$ such that $\forall \tau \in T(n,A)( |\tau| = k \to \exists m < k (\tau(m) > h_{+(n+2)}(m))$ by the previous lemma.
  Define $f$ by
  \begin{align*}
    f(x) = \begin{cases}
      h_{+(n+1)}(x) \text{ if } x < k + 1 \\
      h_{+(n+k +2)}(g (x)) \text{ otherwise.}
    \end{cases}
  \end{align*}
      \setcounter{claimcounter}{1} 
      \begin{claim}
        Now we have
        \begin{enumerate}
          \item $f_{+ (k+1)}$ majorizes $T(n,A)$ but
          \item $f_{+1}$ does not majorize $T(n,A)$.
        \end{enumerate}
      \end{claim}
      \addtocounter{claimcounter}{1}
      \begin{proof of claim}
        (1.) For any $x \in \N$,
          $f_{+(k+1)}(x) = f(x + k + 1)
          = h_{+(n+1)}(g(x))$.
          Since  $h \in [\N]^{\N}$, we have
          $h_{+(n+1)}(g(x)) \geq g(x) \geq g_0(x)$.
          Hence $g_0(x)$ ensures that $f_{+(k+1)}$ majorizes $T(n,A)$.

          (2.) For a contradiction, assume $f_{+1}$ majorizes $T(n,A)$.
          Then there exists $\tau_k \in T(n,A)$ such that
          \begin{equation*}
            |\tau_k| = k \land \forall m < k(\tau_k(m) \leq f_{+1}(m)).
          \end{equation*}
          Now, for any $m < k$,
          \begin{align*}
            f_{+1}(m) &= f(m+1) \\
            &=h_{+(n+1)}(m+1) \\
            &=h(m+n + 2) = h_{+(n+2)}(m).
          \end{align*}
          Therefore, $\forall m < k(\tau_k(m) \leq h_{+(n+2)}(m))$.
          However, this contradicts the definition of $k$.
        \qedclaim
      \end{proof of claim}
      For each $k' \leq k$, define  $g_{k'} \in [\N]^{\N}$ such that
      \begin{align*}
        g_{k'}(0) &= 0, \\
        g_{k'}(1) &= 1, \\
        g_{k'}(x + 2) &= \begin{cases}
          x + k' + n + 2 \text{ if } x + k' + 1 < k +1 \\
          g(x + k' + 1) + (n + k + 2) \text{ otherwise}.
      \end{cases}
      \end{align*}
      Then, for any $k' \leq k$, $g_{k'} \in [\N]^{\N}$ and $f_{+(k'+1)} = (h \circ g_{k'})_{+2}$. Indeed, for any $x$,
      \begin{align*}
        f_{+(k'+1)}(x) &= f(x + k' + 1) \\
        &= \begin{cases}
          h(x + k' + n + 2) \text{ if } x + k' + 1 < k + 1 \\
          h(g(x + k' + 1) + (n + k + 2)) \text{ otherwise}
        \end{cases}
      \end{align*}
      and
      \begin{align*}
        (h \circ g_{k'})_{+2}(x) &= h(g_{k'}(x + 2)) \\
        &= \begin{cases}
          h(x + k' + n + 2) \text{ if } x + k' + 1 < k + 1 \\
          h(g(x + k' + 1) + (n + k + 2)) \text{ otherwise}.
      \end{cases}
      \end{align*}

      Now we have
      \begin{enumerate}
        \item $f_{+(k+1)}$ majorizes $T(n,A)$,
        \item for any $k' \leq k$, if $f_{+(k'+1)}$ majorizes $T(n,A)$ then $f_{+k'}$ majorizes $T(n,A)$.
      \end{enumerate}
      The second condition is proved as follows. We see $n < n+1 \leq h(n+1) = f(0) < f_{k'+1}(0) = (h \circ g_{k'})_{+2}(0)$.
      Therefore, by the assumption on $h$, if $f_{k' + 1} = (h \circ g_{k'})_{+2}$ majorizes $T(n,A)$, then
      $(h \circ g_{k'})_{+1} = f_{k'}$ also majorizes $T(n,A)$.

      By these conditions and arithmetical induction, we have $f_{+1}$ majorizes $T(n,A)$. However, this is a contradiction.
\end{proof}

  \begin{definition}
    ($\ACAo$) For $f \in [\N]^\N$ and $A$,
    define $F^n(f,A)$ by
    \begin{align*}
      F^0(f,A) &= A, \\
      F^{n+1}(f,A) &= F(f,F^{n}(f,A)).
    \end{align*}
  \end{definition}

  \begin{definition}
    Define an arithmetical formula $\Phi_n(f,A)$ by
    \begin{align*}
      \Phi_n(f,A) \equiv \bigwedge_{i < n} \Phi_{\ast}(f,F^i(f,A)).
    \end{align*}
  \end{definition}

\begin{lemma}\label{oneside for n}
  Let $k \in \omega$. Then $\ACAo$ proves that
  for any $h \in [\N]^{\N}$ and $A$, if $\RP_{\Phi_k}(h,A)$ holds, then
  then $(\forall g \Phi_k(h \circ g,A))$ holds.
\end{lemma}
\begin{proof}
  Take $h$ and $A$ as above.
  For the sake of contradiction, assume $(\forall g \lnot \Phi_k(h \circ g,A))$.

  For each $n$ define $g_n$ by $g_n(0) = 0, g_n(x+1) = n+x+1$.
  Then, since $\lnot \Phi(h \circ g, A)$ holds, there exists $m < h(g(0)) = h(0)$ such that
  \begin{align*}
    \bigvee_{i < k} (h \circ g_n)_{+2} \text{ majorizes } T(m,F^i((h \circ g_n)_{+2},A)) \text{, but } (h \circ g_n)_{+1} \text{ does not. }
  \end{align*}
  We note that $(h \circ g_n)_{+1}(x) = (h \circ g_n)(x + 1) = h(g_n(x+1)) = h(n + x + 1) = h_{+(n+1)}(x)$ and
  $(h \circ g_n)_{+2}(x) = h_{+(n+2)}(x)$.
  Hence, for each $n$, there exists $m < h(0)$ such that
  \begin{align*}
    \bigvee_{i < k} h_{+(n+2)} \text{ majorizes } T(m,F^i(h_{+(n+2)},A)) \text{ but } h_{+(n+1)} \text{ does not. }
  \end{align*}
  \setcounter{claimcounter}{1} 
  \begin{claim}
    For each $i < k$ and $n$, $F^i(h_{+(n+2)},A) = F^i(h,A)$.
  \end{claim}
  \addtocounter{claimcounter}{1}
  \begin{proof of claim}
    We will show by (meta-)induction on $i$.

    The base step is immediate. By definition, $F^0(h_{+(n+2)},A) = A = F^0(h,A)$.

    For the induction step, assume $F^i(h_{+(n+2)},A) = F^i(h,A)$.
    Then
    \begin{align*}
      F^{i+1}(h_{+(n+2)},A) &= F(h_{+(n+2)},F^{i}(h_{+(n+2)},A)) \\
      &= F(h_{+(n+2)},F^i(h,A)) \\
      &= \{x : \exists p ((h_{+(n+2)})_{+p} \text{ majorizes } T(x,F^i(h,A))) \} \\
      &= \{x : \exists p (h_{+p} \text{ majorizes } T(x,F^i(h,A))) \} = F(h,F^i(h,A)).
    \end{align*}
    This completes the proof.
    \qedclaim
  \end{proof of claim}
  Now, for each $n$, there exists $m < h(0)$ such that
  \begin{align*}
    \bigvee_{i < k } h_{+(n+2)} \text{ majorizes } T(m,F^i(h,A)) \text{ but } h_{+(n+1)} \text{ does not. }
  \end{align*}
  Put $\theta(i,n) \equiv h_{+(n+2)} \text{ majorizes } T(m,F^i(h,A)) \text{ but } h_{+(n+1)} \text{ does not }$.

  For each $n$, define $m(n)$ be the smallest $m < h(0)$ satisfying the above condition.
  \begin{claim}
  For each $m$, there are at most $k$-many $n$ such that $m(n) = m$.
  \end{claim}
  \addtocounter{claimcounter}{1}
  \begin{proof of claim}
    Fix an $m$. For the sake of contradiction, let $n_0,\ldots,n_k$ be
    such that  $n_0 < \cdots < n_k$ and $m(n_j) = m$ for all $j < k$.
    Then, for each $j < k$, there exists an $i_j < k$ such that $\theta(i_j,n_j)$ holds.

      We claim that if $j \neq j'$ then $i_j \neq i_{j'}$. Let $j < j' \leq k$.
      If $i_j = i_{j'}$, then
      \begin{align*}
         &h_{+(n_j+2)} \text{ majorizes } T(m,F^{i_j}(h,A)) \text{ but } h_{+(n_j+1)} \text{ does not } \ \land \\
         &h_{+(n_{j'}+2)} \text{ majorizes } T(m,F^{i_j}(h,A)) \text{ but } h_{+(n_{j'}+1)} \text{ does not }.
      \end{align*}
      However, since $n_j < n_{j'}$ and $h_{+(n_j+2)}$ majorizes  $T(m,F^{i_j}(h,A))$,
      $h_{+(n_{j'}+1)}$ should majorize $T(m,F^{i_j}(h,A))$. This is a contradiction.

    Now $\{i_j : j \leq k\}$ should be a $k+1$-element subset of $\{0,\ldots,k-1\}$. This is a contradiction.
  \qedclaim
  \end{proof of claim}
  Now, $m(n)$ is a total function such that $\{m(n) : n \in \N\}$ is bounded by $h(0)$.
  However, for any $y \in \{m(n) : n \in \N\}$, there are at most $k$-many $n$ such that $m(n) = y$. This is a contradiction.
\end{proof}

\begin{lemma}\label{lightface ram to n jump}
  ($\ACAo$)
  For any $h \in [\N]^{\N}$ and $A$, if $\forall g \Phi_n(h \circ g,A)$ then
  $F^i(h,A) = \HJ^i(A)$ for any $i \leq n$.
\end{lemma}
\begin{proof}
  Take $h$ and $A$ as above.
  Then, for the formula $\Phi_{\ast}$ in Definition \ref{def of phi-star}, $\forall g (\Phi_{\ast}(h \circ g,A))$ holds.
  Hence $F(h,A) = \HJ(A)$.
  Thus, $\forall g(\Phi_{\ast}(h \circ g,\HJ(A)))$ also holds, and hence $F(h,F(h,A)) = F(h,\HJ(A)) = \HJ^2(A)$.
  Iterating this argument, we have $F^i(h,A) = \HJ^i(A)$ for any $i \leq n$.
\end{proof}

\begin{theorem}\label{Ram and hyperjumps}
  Over $\ACA^+_0$, $\{\rel (\Sigma^0_n \Ram) : n \in \omega\}$ and $\{\beta^1_0\RFN(n) : n \in \omega\}$ prove the
  same sentences.
  Therefore, any $\Pi^1_2$ sentence provable from $\Pi^1_1\myhyphen\CAo$ is provable from
  $\ACA^+_0 + \rel (\Sigma^0_n \Ram)$ for some $n \in \omega$.
\end{theorem}
\begin{proof}
  As we have proved in Theorem \ref{betarfn to ram}, $\beta^1_0\RFN(n)$ implies $\rel(\Sigma^0_n\Ram)$ for each $n$.

  We show that $\beta^1_0\RFN(n)$ is provable from $\{\rel (\Sigma^0_n\Ram) : n \in \omega\}$.
  Assume $\ACA^+_0 + \{\rel \Sigma^0_n \Ram : n \in \omega\}$. Take an arbitrary $n \in \omega$.
  Let $X$ be a set. We will show that there exists a coded $\omega$-model $\M$ such that
  $X \in \M$ and $\M \models \exists Y(Y = \HJ^n(X))$.

  Now $\Phi_n(f,X)$ is an arithmetical formula for $f \in [\N]^{\N}$.
  Therefore, there exists an $h \in [\N]^{\N}$ such that
  $\pRP_{\Phi_n}(h,X)$ holds.
  By $\ACA_0^+$, take  $\M = \{Y : Y \leq_{\T}^{\a} h \oplus X\}$. Then $\M \models \ACAo + (\forall g\Phi_n(h \circ g, X)) \lor (\forall g \lnot \Phi_n(h \circ g,X))$. Thus, by Lemma \ref{oneside for n},
  $\M \models \ACAo + \forall g \Phi_n(h \circ g, X)$.
  Therefore, by Lemma \ref{lightface ram to n jump}, $\M \models \exists Y(Y = \HJ^n(X))$.
\end{proof}

  In this section, we have proved that
  \begin{itemize}
    \item for any $n \in \omega$, $\beta^1_0\RFN(n)$ implies $\rel(\Sigma^0_n\Ram)$, and
    \item for any $n \in \omega$, there is some $m \in \omega$ such that $\rel(\Sigma^0_m\Ram)$ implies $\beta^1_0\RFN(n)$
  \end{itemize}
  over $\ACA_0^+$.
  Therefore, it is natural to ask
  \begin{question}
    Is $\beta^1_0\RFN(n)$ provable from $\rel(\Sigma^0_{n+1}\Ram)$ over $\ACA_0^+$ ?
\end{question}
A precise analysis of relationship between hyperjumps and Ramsey's theorem is studied in \cite{Marcone_Ramsey} in the context of Weihrauch degrees.
It is expected a natural formalization of this work would give a positive answer for this question.

\section{Pseudo determinacy and $\beta^1_0\RFN$}
In this section, we introduce a $\Pi^1_2$ variant of determinacy which we call pseudo-determinacy.
Then we consider the hierarchy of pseudo-determinacy for $(\Sigma^0_1)_k$ games.
First of all, we introduce the notion of difference hierarchy and the determinacy.

\begin{definition}[Difference hierarchy]
  Let $k \in \omega$.
  We define the class $(\Sigma^0_1)_k$ of formulas as follows:
  \begin{itemize}
    \item a formula is $(\Sigma^0_1)_1$ if it is $\Sigma^0_1$,
    \item a formula $\varphi$ is $(\Sigma^0_1)_{k+1}$ if there are a $\Sigma^0_1$ formula $\psi$ and a $(\Sigma^0_1)_{k}$ formula $\theta$ such that $\varphi \equiv \psi \land \lnot \theta$.
  \end{itemize}
  We call the hierarchy formed by $\{(\Sigma^0_1)_k : k \in \omega\}$ the difference hierarchy of $\Sigma^0_1$ formulas.
\end{definition}

\begin{definition}\label{definition of determinacy}
  We call a function $S : \N^{<\N} \to \N$ a strategy\footnote{Usually, a strategy is a function defined over
  $\bigcup_{n} \N^{2n}$ or $\bigcup_{n} \N^{2n+1}$.}.
  Let $S$ and $S'$ be strategies. We say a function $f \in \N^{\N}$ is the play along $(S,S')$ (write $f = S \otimes S'$) if
  \begin{equation*}
      \forall n (f(2n) = S(f[2n]) \land f(2n+1) = S'(f[2n+1])).
  \end{equation*}

  Let $\varphi(\vec{X},f)$ be a formula for $f \in \N^{\N}$.
  Let $\vec{X}$ be sets.
  We say $\varphi(\vec{X},\bullet)$ is determined if
  \begin{align*}
    \exists S_0 \forall S_1 \varphi(\vec{X},S_0 \otimes S_1) \lor
    \exists S_1 \forall S_0 \lnot\varphi(\vec{X},S_0 \otimes S_1).
  \end{align*}
  Especially, if a strategy $S_0$ satisfies
  \begin{align*}
    \forall S_1 \varphi(\vec{X},S_0 \otimes S_1),
  \end{align*}
  then we say $\varphi(\vec{X},\bullet)$ is determined via a strategy $S_0$ for player $0$.
  If a strategy $S_1$ satisfies
  \begin{align*}
    \forall S_0 \lnot \varphi(\vec{X},S_0 \otimes S_1),
  \end{align*}
  then we say $\varphi(\vec{X},\bullet)$ is determined via a strategy $S_1$ for player $1$.

  We also say $\varphi$ is determined if for any $\vec{X}$, $\varphi(\vec{X},\bullet)$ is determined.
  Let $\Gamma$ be a class of formulas for $f \in \N^{\N}$. We say $\Gamma$ is determined (write $\Gamma\myhyphen\Det$) if any $\varphi \in \Gamma$ is determined.
\end{definition}

It is well-known that over $\RCAo$,
\begin{itemize}
  \item $\ATRo$ is equivalent to $\Sigma^0_1\myhyphen\Det$, and
  \item $\Pi^1_1\myhyphen\CAo$ is equivalent to $(\Sigma^0_1)_k\myhyphen\Det$ for $k > 1$.
\end{itemize}

In this section, we consider  $\Pi^1_2$ variants of the second clause.
\begin{definition}[pseudo-determinacy]\label{definition of rel det}
  Let $\varphi(\vec{X},f)$ be a formula for $f \in \N^{\N}$.
  For given sets $\vec{X}$, we say $\varphi(\vec{X},\bullet)$ is pseudo-determined if
  \begin{align*}
    \exists S_0 \forall S_1 \leq_{\T}^a S_0 \oplus \vec{X} \varphi(\vec{X},S_0 \otimes S_1) \lor
    \exists S_1 \forall S_0 \leq_{\T}^a S_1 \oplus \vec{X} \lnot\varphi(\vec{X},S_0 \otimes S_1)).
  \end{align*}
  If $\varphi(\vec{X},\bullet)$ is pseudo-determined via $S$, then we say $S$ is a pseudo-winning strategy for the game defined by $\varphi$ and $\vec{X}$. If $\vec{X}$ is clear from the context, we just say $S$ is a pseudo-winning strategy for the game defined by $\varphi$.
  If $\varphi(\vec{X},\bullet)$ is pseudo-determined for any $\vec{X}$, then we say $\varphi$ is pseudo-determined.

  Let $\Gamma$ be a class of formulas for $f \in \N^{\N}$. We say $\Gamma$ is pseudo-determined if any $\varphi \in \Gamma$ is pseudo-determined.
  Let $\rel(\Gamma\myhyphen\Det)$ denote the assertion that every game defined by a formula in $\Gamma$ is pseudo-determined.
\end{definition}
  We note that the statement [$\varphi$ is pseudo-determined] is of the form $\rel(\psi)$ for some $\psi$ in the sense of Section 5.
  Moreover, each $\rel((\Sigma^0_1)_n\myhyphen\Det)$ is axiomatizable by a $\Pi^1_2$ sentence of the form $\rel(\psi)$.

\begin{remark}
  For any $(\Sigma^0_1)_k$ formula $\varphi(X_1,\ldots,X_n,f)$, there is a $(\Sigma^0_1)_k$ formula $\psi(X,f)$ such that $\RCAo \vdash \forall X_1,\ldots,X_n,f(\varphi(X_1,\ldots,X_n,f) \leftrightarrow \psi(X_1 \oplus \cdots \oplus X_n,f)$. Therefore, when considering $\rel (\Sigma^0_1)_k\myhyphen\Det$, we may assume that any $(\Sigma^0_1)_k$
  formula has exactly one variable $X$ and one function variable $f$.

  Let $\varphi(X,f)$ be a $(\Sigma^0_1)_k$ formula. We may assume that
  for any $X$, all strategies for the game $\varphi(X,\bullet)$ computes $X$.
  Indeed, if $S$ is a strategy for player $0$, then $S$ is enough to be defined at sequences with even length. Thus, we may assume that $X(n) = S(1^{2n+1})$. Similarly, if $S$ is a strategy for player $1$, then we may assume that $X(n) = S(1^{2n})$.
  Therefore, $\varphi(X,f)$ is pseudo-determined if
  \begin{align*}
    \forall X(
    \exists S_0 \forall S_1 \leq_{\T}^a S_0 \varphi(X,S_0 \otimes S_1) \lor
    \exists S_1 \forall S_0 \leq_{\T}^a S_1 \lnot\varphi(X,S_0 \otimes S_1)).
  \end{align*}
\end{remark}

We note that for any $k \in \omega, k > 1$, $\rel (\Sigma^0_1)_k\myhyphen\Det$ is a $\Pi^1_2$ statement provable from $\Pi^1_1\myhyphen\CAo$. Thus, each $\rel (\Sigma^0_1)_k\myhyphen\Det$ is provable from some $\beta^1_0\RFN(n)$ over $\ACAo$.
We  show that $\beta^1_0\RFN(k;\ATRo)$ is enough to prove $\rel (\Sigma^0_1)_k\myhyphen\Det$.

To prove $(\Sigma^0_1)_2$ determinacy from $\Pi^1_1\myhyphen\CAo$, Tanaka proved the following lemma.
\begin{lemma}
  Let $\varphi(X,f)$ be a $(\Sigma^0_1)_2$ formula with exactly displayed free set variables.
  Assume $\ATRo$.
  For any $X$, if $\HJ(X)$ exists, then $\varphi(X,\bullet)$ is determined.
\end{lemma}
\begin{proof}
  See \cite{tanaka1990weak}.
\end{proof}
We extend this lemma to $(\Sigma^0_1)_k$ formulas.
\begin{lemma}
  Let $k \in \omega$ and $\varphi(X,f)$ be a $(\Sigma^0_1)_{k+1}$ formula with exactly displayed set variables.
  Over $\ATRo$, for any set $X$, if $\HJ^k(X)$ exists, then $\varphi(X,\bullet)$ is determined.
\end{lemma}
\begin{proof}
  We will show by induction on $k \in \omega$.
  The case $k=0$ is trivial because $\ATRo$ proves $\Sigma^0_1$ determinacy.

  Let $k \in \omega$. Assume that $\ATRo$ proves that for any $X$, if $\HJ^k(X)$ exists then any $(\Sigma^{0,X}_1)_{k+1}$ game is determined.

  We reason in $\ATRo$. Let $X$ be a set such that $\HJ^{k+1}(X)$ exists. Let $\varphi(A,f)$ be a $(\Sigma^0_1)_{k+2}$ formula.
  Then, there is a $\Sigma^0_1$ formula $\psi(A,f)$ and a $(\Sigma^0_1)_{k+1}$ formula $\theta(A,f)$ such that
  $\varphi(A,f) \equiv \psi(A,f) \land \lnot \theta(A,f)$. Moreover, $\psi(A,f)$ is written as $\exists n \eta(A,f[n])$ for a $\Sigma^0_0$ formula $\eta$.
  Take a coded $\beta$ model $\M$ such that
  $X \in \M \models \exists Y(Y = \HJ^k(X))$.
  Define $W$ by
  \begin{align*}
    W = \{ \sigma \in \N^{<\N} : \eta(X,\sigma) \land \M \models \text{player 0 wins for } \lnot \theta^{\sigma}(X,\bullet) \}.
  \end{align*}
  Here, $\lnot \theta^{\sigma}(X,f)$ is the formula $\lnot \theta(X,\sigma \ast f)$ and $\sigma \ast f$ denotes the  concatenation of $\sigma$ and $f$.
  Let $\varphi'(f) \equiv \exists n (f[n] \in W)$. Then, the game defined by $\varphi'$ is determined.

  Case 1: player 0 wins the game $\varphi'$ via $S_0$.
  Now, for any $\sigma \in W$, $\M \models \exists S_0^{\sigma} \forall S_1^{\sigma} \lnot \theta(X,\sigma \ast S_0^{\sigma} \otimes S_1^{\sigma})$. Take a sequence $\langle S_0^{\sigma}\rangle_{\sigma \in W}$.
  Then, each $S_0^{\sigma}$ satisfies $\forall S_1^{\sigma} \lnot \theta(X,\sigma \ast S_0^{\sigma} \otimes S_1^{\sigma})$ because $\M$ is a $\beta$-model.
  Consider a strategy for player 0 such that
  \begin{itemize}
    \item while the play is not in $W$, the strategy mimics $S_0$,
    \item once the play $\sigma$ is in $W$, the strategy mimics $S_0^\sigma$.
  \end{itemize}
  Then this strategy is a winning strategy for player 0.

Case 2: player 1 wins the game $\varphi'$ via $S_1$.
Now we have that for each play $\sigma$ consistent with $S_1$, if $\eta(X,\sigma)$ then
\begin{align*}
  \M \models \text{player 0 has no winning strategy for } \lnot \theta^{\sigma}(X,\bullet).
\end{align*}
By induction hypothesis, $\M \models \theta^{\sigma}(X,\bullet) \text{ is determined}$. Hence for each $\sigma$ consistent with $S_1$, if $\eta(X,\sigma)$ then
$\M \models \exists S_1^{\sigma} \forall S_0^{\sigma} \theta(X,\sigma \ast S_0^{\sigma} \otimes S_1^{\sigma})$.
Take a sequence $\langle S_1^{\sigma} : \sigma \text{ is consistent with } S_1 \land \eta(X,\sigma) \rangle$ such that
$\M \models \forall S_0^{\sigma} \theta(X,\sigma \ast S_0^{\sigma} \otimes S_1^{\sigma})$.
Then, since $\M$ is a coded $\beta$-model, for each $S_0^{\sigma}$, $\forall S_1^{\sigma} \theta(X,\sigma \ast S_0^{\sigma} \otimes S_1^{\sigma})$ holds.
Consider a strategy for player 1 such that
\begin{itemize}
  \item while the play $\sigma$ does not satisfy $\eta(X,\sigma)$, the strategy mimics $S_1$,
  \item once the play $\sigma$ satisfies $\eta(X,\sigma)$, the strategy mimics $S_1^\sigma$.
\end{itemize}
Then this strategy is a winning strategy for player 1.
\end{proof}

%
%

\begin{theorem}\label{beta-rfn-atr0 and det}
  For each $k \in \omega, k > 0$, $\beta^1_0\RFN(k,\ATRo)$ proves $\rel(\Sigma^0_1)_{k+1}\Det$.
\end{theorem}
\begin{proof}
  Let $\varphi(A,f)$ be a $(\Sigma^0_1)_{k+1}$ formula.

  Assume $\beta^1_0\RFN(k,\ATRo)$. Let $X$ be a set.
  By $\beta^1_0\RFN(k,\ATRo)$ take a coded $\omega$-model $\M$ such that
  $X \in \M \models \ATRo + \exists Y(Y = \HJ^k(X))$.
  Then, by the previous lemma, $\M \models \varphi(X,\bullet) \text{ is determined}$.
  Without loss of generality, we may assume that
  $\M \models \text{player 0 wins for $\varphi(X,\bullet)$}$.
  Let $S_0 \in \M$ be a winning strategy of player $0$. Then, for any $S_1 \leq_{\T}^{\a} S_0$, $S_1 \in \M$ and hence $\varphi(X,S_0 \otimes S_1)$.
\end{proof}

We next see the implication from pseudo-determinacy to $\beta^1_0\RFN$.
First we see the implication from $\rel (\Sigma^0_1)_2\myhyphen\Det$.
\begin{lemma}\label{transfinite jump of hyperjump}
  There exists a $(\Sigma^0_1)_2$ formula $\varphi(f,X,Y)$ such that $\ACAo$ proves the following.
  \begin{equation*}
    \forall \alpha : \WO \forall X (\text{$\varphi(\bullet,X,\alpha)$ is determined} \to \exists Y(Y = (\HJ(X))^{(\alpha)}).
  \end{equation*}
\end{lemma}
\begin{proof}
  See \cite{tanaka1990weak}.
\end{proof}

\begin{theorem}\label{rel-det and TLPP}
  Over $\ACA_0^+$, $\rel (\Sigma^0_1)_2\myhyphen\Det$ proves $\TLPP$.
\end{theorem}
\begin{proof}
  By \ref{characterization of TLPP by reflection}, it is enough to show that
  \begin{align*}
    \forall \alpha : \WO \forall X \exists \M
    (&\alpha,X \in \M, \\
    &\M \models \ACAo + \exists Y(Y= \HJ(X)), \\
    &\text{$\M$ is closed under $\alpha$-jump}).
  \end{align*}

  Let $\varphi(f,X,Y)$ be the formula in the Lemma \ref{transfinite jump of hyperjump}.
  Take a well-ordering $\alpha$ and a set $X$.
  By $\rel (\Sigma^0_1)_2\myhyphen\Det$ take a pseudo-winning strategy $S$ for the game $\varphi(\bullet,X,\alpha \cdot \omega)$.
  Let $\M$ be a coded $\omega$-model of $\ACAo$ containing $X,\alpha \cdot \omega,S$.
  Then $\M \models \exists Y(Y = (\HJ(X))^{(\alpha \cdot \omega)})$.

  We define a coded $\omega$-model $\M'$ by
  \begin{equation*}
    \M' = \{A \in \M : \M \models \exists n(A \leq_{\T} (\HJ(X))^{(\alpha \cdot n)})\}.
  \end{equation*}
  Then, $\M'$ is a coded $\omega$-model of $\ACAo$ closed under $\alpha$-jump.
  Moreover, since $\M' \in \M$ and $\M \models [\HJ(X) \in \M']$, $\M' \models \exists Y(Y = \HJ(X))$ by Lemma \ref{Hyperjump in the ground model is also hyperjump in a coded model}.
\end{proof}
\begin{corollary}
  $\beta^1_0\RFN(1)$ is not enough to prove $\rel (\Sigma^0_1)_2\myhyphen\Det$.
\end{corollary}
\begin{proof}
  It follows from $\beta^1_0\RFN(1) < \TLPP \leq \rel (\Sigma^0_1)_2\myhyphen\Det$.
\end{proof}
\begin{remark}
  In his paper \cite{tanaka1990weak}, Tanaka pointed that $\ACAo$ is not strong enough to prove ``$\Pi^1_1$ comprehension implies $(\Sigma^0_1)_2\myhyphen\Det$" by using the axiomatic system lightface $\Pi^1_1$ comprehension.
  The previous corollary is another expression of this fact.
\end{remark}

At last, we see the general case.
The point is the following lemma.

\begin{lemma}
  Let $n > 0$.
  Then there is a $(\Sigma^0_1)_{p(n)}$ formula $G_n(X,f)$ such that $\ACAo$ proves that
  \begin{itemize}
    \item $\forall X (G_n(X,\bullet) \text{ is determined} \to \HJ^n(X) \text{ exists.})$
    \item Any winning strategy for $G_n$ computes $X$.
  \end{itemize}
  Here, $p(n)$ is a certain primitive recursive function.
\end{lemma}
\begin{proof}
  See the proof of Proposition 6 in \cite{pacheco2022determinacy}.
\end{proof}

\begin{theorem}
  Let $n > 0$. Then, over $\ACA_0^+$,
  $\rel (\Sigma^0_1)_{p(n)}\myhyphen\Det$ implies $\beta^1_0\RFN(n)$.
\end{theorem}
\begin{proof}
  We reason in $\ACA_0^+ + \rel (\Sigma^0_1)_{p(n)}\myhyphen\Det$.
  Let $X$ be a set.
  By $\rel (\Sigma^0_1)_{p(n)}$, take a pseudo-winning strategy $S$ for the game $G_n(X,\bullet)$.
  Let $\M$ be an $\omega$-model of $\ACAo$ containing $S$.
  Then $\M \models [G_n(X,\bullet) \text{ is determined via }S]$, and hence $\M \models \exists Y(Y = \HJ^n(X))$.
  This completes the proof.
\end{proof}

\begin{corollary} \label{Det and hyperjumps}
  Over $\ACA^+_0$, $\{\rel ((\Sigma^0_1)_n\myhyphen\Det) : n \in \omega\}$ and $\{\beta^1_0\RFN(n) : n \in \omega\}$ prove the
  same sentences.
  Therefore, any $\Pi^1_2$ sentence provable from $\Pi^1_1\myhyphen\CAo$ is provable from
  $\ACA^+_0 + \rel ((\Sigma^0_1)_n\myhyphen\Det))$ for some $n \in \omega$.
\end{corollary}

\begin{remark}
  We have proved in Lemma \ref{beta(n+1) is strictly stronger than beta(n)}, Theorem \ref{beta-rfn-atr0 and det} and \ref{rel-det and TLPP} that
  \begin{itemize}
    \item $\TLPP \leq \rel((\Sigma^0_1)_2\myhyphen\Det) \leq \beta^1_0\RFN(1;\ATRo)$,
    \item $\TLPP < \beta^1_0\RFN(1;\ATRo)$.
  \end{itemize}
  Thus, at least one of $\TLPP$ or $\beta^1_0\RFN(1;\ATRo)$ is not equivalent to $\rel((\Sigma^0_1)_2\myhyphen\Det)$.
\end{remark}
\begin{question}
  Can we separate $\rel((\Sigma^0_1)_2\myhyphen\Det)$ and $\TLPP$ or $\beta^1_0\RFN(1;\ATRo)$?
\end{question}


\bibliographystyle{plain}
\bibliography{references}

\begin{thebibliography}{10}

\bibitem{book:hb-pfthy}
Samuel~R. Buss, editor.
\newblock {\em Handbook of Proof Theory}, volume 137.
\newblock Amsterdam, 1998.

\bibitem{DzMu}
Damir~D. Dzhafarov and Carl Mummert.
\newblock {\em Reverse mathematics---problems, reductions, and proofs}.
\newblock Theory and Applications of Computability. Springer, Cham, [2022]
  \copyright 2022.

\bibitem{F-DSTY}
D.~Fern\'andez-Duque, P.~Shafer, H.~Towsner, and K.~Yokoyama.
\newblock Metric fixed point theory and partial impredicativity.
\newblock {\em Philos. Trans. Roy. Soc. A}, 381(2248):Paper No. 20220012, 20,
  2023.

\bibitem{Freund-Fraisse}
Anton Freund.
\newblock Fra\"iss\'e's conjecture, partial impredicativity and well-ordering
  principles, part i.
\newblock {\em arXiv preprint arXiv:2406.13485}, 2024.

\bibitem{Marcone_Ramsey}
Gian Marco and Albert Marcone.
\newblock The {G}alvin-{P}rikry theorem in the weihrauch lattice.
\newblock {\em arXiv preprint arXiv:2410.06928}.

\bibitem{marcone_leftmost}
Alberto Marcone.
\newblock On the logical strength of {N}ash-{W}illiams' theorem on transfinite
  sequences.
\newblock In {\em Logic: from foundations to applications: European logic
  colloquium}, pages 327--351, 1996.

\bibitem{Montalban-Fraisse}
Antonio Montalb\'an.
\newblock Fra\"iss\'e's conjecture in {$\Pi_1^1$}-comprehension.
\newblock {\em J. Math. Log.}, 17(2):1750006, 12, 2017.

\bibitem{montalban2018conservativity}
Antonio Montalb{\'a}n and Richard~A Shore.
\newblock Conservativity of ultrafilters over subsystems of second order
  arithmetic.
\newblock {\em The Journal of Symbolic Logic}, 83(2):740--765, 2018.

\bibitem{pacheco2022determinacy}
Leonardo Pacheco and Keita Yokoyama.
\newblock Determinacy and reflection principles in second-order arithmetic.
\newblock {\em arXiv preprint arXiv:2209.04082}, 2022.

\bibitem{pakhomov_solda_nash}
Fedor Pakhomov and Giovanni Sold{\`a}.
\newblock On {N}ash-{W}illiams' theorem regarding sequences with finite range.
\newblock {\em arXiv preprint arXiv:2405.13842}, 2024.

\bibitem{incoll:pflen}
Pavel Pudl{\'a}k.
\newblock The lengths of proofs.
\newblock In Buss \cite{book:hb-pfthy}, chapter VIII, pages 547--637.

\bibitem{Shafer_Menger}
Paul Shafer.
\newblock Menger's theorem in {{\(\Pi^1_1 \mathrm {-CA}_0\)}}.
\newblock {\em Arch. Math. Logic}, 51(3-4):407--423, 2012.

\bibitem{Simpson}
Stephen~G. Simpson.
\newblock {\em Subsystems of second order arithmetic}.
\newblock Perspectives in Logic. Cambridge University Press, Cambridge;
  Association for Symbolic Logic, Poughkeepsie, NY, second edition, 2009.

\bibitem{suzuki2024relative}
Yudai Suzuki.
\newblock Relative leftmost path principles and omega-model reflections of
  transfinite inductions.
\newblock {\em arXiv preprint arXiv:2407.13504}, 2024.

\bibitem{suzuki_yokoyama_fp}
Yudai Suzuki and Keita Yokoyama.
\newblock Searching problems above arithmetical transfinite recursion.
\newblock {\em Annals of Pure and Applied Logic}, page 103488, 2024.

\bibitem{Tanaka-Ramsey}
Kazuyuki Tanaka.
\newblock The {G}alvin-{P}rikry theorem and set existence axioms.
\newblock {\em Annals of Pure and Applied Logic}, 42(1):81--104, 1989.

\bibitem{tanaka1990weak}
Kazuyuki Tanaka.
\newblock Weak axioms of determinacy and subsystems of analysis {I}:
  {$\Delta^0_2$} games.
\newblock {\em Mathematical Logic Quarterly}, 36(6):481--491, 1990.

\bibitem{Townser_TLPP}
Henry Towsner.
\newblock Partial impredicativity in reverse mathematics.
\newblock {\em J. Symbolic Logic}, 78(2):459--488, 2013.

\end{thebibliography}

\end{document}